\documentclass{article}
\usepackage{amssymb,latexsym,amsmath,amsthm}

\newtheorem{theorem}{Theorem}[section]
\newtheorem{lemma}[theorem]{Lemma}
\newtheorem{corollary}[theorem]{Corollary}
\newtheorem{proposition}[theorem]{Proposition}

\theoremstyle{definition}
\newtheorem{definition}[theorem]{Definition}

\newtheorem{remark}[theorem]{Remark}

\newcommand{\End}{{ \rm End }}

\newcommand{\Cok}{{ \rm Cok }}

\newcommand{\Ext}{{ \rm Ext }}
\newcommand{\Hom}{{ \rm Hom }}
\newcommand{\Ker}{{ \rm Ker }\,}

\newcommand{\Mod}{{ \rm Mod }}

\newcommand{\Rad}{{ \rm Rad }}
\newcommand{\Ann}{{ \rm Ann }}

\newcommand{\reg}{{ \rm reg }}

\newcommand{\rad}{{ \rm rad }}

\newcommand{\Q}{{\mathtt{Q}}}
\newcommand{\LQ}{{\mathtt{L}}}

\newcommand{\soc}{{ \rm soc }}

\newcommand{\g}{\hbox{-}}
\newcommand{\uddots}{\mathinner{\mkern1mu\raise1pt\vbox{\kern7pt\hbox{.}}
\mkern2mu\raise4pt\hbox{.}\mkern2mu\raise7pt\hbox{.}\mkern1mu}}
\newcommand{\Endol}[1]{\hbox{\rm endol}\,(#1)}

\newcommand{\hueca}[1]{\mathbb{#1}}
\newcommand{\lddots}{
\mathinner{
\mkern1mu\raise1pt}\vbox{\kern7pt\hbox{.}}
\mkern2mu\raise3pt\hbox{.}
\mkern2mu\raise7pt\hbox{.}\mkern1mu}

\renewcommand{\mod}{{\rm mod }}
\renewcommand{\Im}{{\rm Im}\,}

\newcommand{\rightdashmap}[1]{\smash{\mathop{\hbox to 
20pt{-\,-\,-\,\rightarrowfill}}\limits^{#1}}}

\newcommand{\rightmap}[1]{\smash{\mathop{\hbox to 
20pt{\rightarrowfill}}\limits^{#1}}}
\newcommand{\leftmap}[1]{\smash{\mathop{\hbox to 
20pt{\leftarrowfill}}\limits^{#1}}}

\newcommand{\longrightmap}[1]{\smash{\mathop{\hbox to 
4cm{\rightarrowfill}}\limits^{#1}}}
\newcommand{\longleftmap}[1]{\smash{\mathop{\hbox to 
4cm{\leftarrowfill}}\limits^{#1}}}
\newcommand{\medrightmap}[1]{\smash{\mathop{\hbox to 
2cm{\rightarrowfill}}\limits^{#1}}}
\newcommand{\medleftmap}[1]{\smash{\mathop{\hbox to 
2cm{\leftarrowfill}}\limits^{#1}}}

\newcommand{\shortlmapdown}[1]
{\llap{$\vcenter{\hbox{$\scriptstyle#1$}}$}\big\downarrow}
\newcommand{\shortrmapdown}[1]
{\big\downarrow\rlap{$\vcenter{\hbox{$\scriptstyle#1$}}$}}

\newcommand{\idmapdown}[1]
{\hskip-8pt\mathop{\hskip-5pt\raise6pt
\hbox{$\scriptstyle#1$}\hskip-5pt\swarrow}}

\newcommand{\ddmapdown}[1]
{\hskip-5pt\mathop{\searrow\hskip-6pt\raise5pt\hbox{$\scriptstyle#1$}}}

\newcommand{\idmapup}[1]
{\hskip-5pt\mathop{\nwarrow\hskip-6pt \raise5pt\hbox{$\scriptstyle#1$}}}

\newcommand{\ddmapup}[1]
{\hskip-8pt\mathop{\hskip-5pt\raise6pt\hbox{$\scriptstyle#1$}\hskip-5pt\nearrow}
}

\newcommand{\flechypunt}[2]{\ \smash{\mathop{
   \raise 3pt \hbox to 40pt{\rightarrow}\hskip-40pt \lower 3pt
   \hbox to 40pt{\dashrightarrow}}\limits^{#1}_{#2}}\ }

\newcommand{\longequal}{\ \smash{\mathop{
   \raise 5pt \hbox to 35pt{\hrulefill}\hskip-35pt \lower 0pt
   \hbox to 35pt{\hrulefill}}}\ }

\newcommand{\raya}[1]{\ \smash{\mathop{\raise 2pt \hbox to 
10pt{\hrulefill}}\limits^{#1}}\ }

\title{\bf {\Large A representation embedding for algebras of infinite type }}

\author{R. Bautista, E. P\'erez, L. Salmer\'on}

\begin{document}
  \date{} 
 \maketitle
 
   \begin{abstract}
\noindent 
   We show that for any finite-dimensional algebra $\Lambda$ of infinite representation type, over a perfect field, there is a  bounded principal ideal domain $\Gamma$ and a representation embedding from $\Gamma\g\mod$ into $\Lambda\g\mod$. As an application, we prove a variation of the Brauer-Trall Conjecture II: finite-dimensional algebras of infinite-representation type admit infinite  families of non-isomorphic finite-dimensional indecomposables with fixed endolength, for infinitely many endolengths. 
\end{abstract}

\renewcommand{\thefootnote}{}



  \section{Introduction}
  
  In this note, \emph{$k$ denotes a perfect field, which will act centrally on every algebra or bimodule we consider}. Given any $k$-algebra  $\Lambda$, we denote by $\Lambda\g\Mod$ (resp. $\Lambda\g\mod$) the category of left $\Lambda$-modules (resp. finite-dimensional left $\Lambda$-modules). We denote by $\Lambda\g\mod_{\cal I}$ the full subcategory of $\Lambda\g\mod$
 formed by the modules without injective direct summands.
 
 Recall that the \emph{endolength} of a $\Lambda$-module $M$, denoted by 
 $\Endol{M}$, is by definition the length of the right $\End_{\Lambda}(M)^{op}$-module $M$. We say that $\Lambda$ is \emph{discrete} (resp.
\emph{$e$-discrete}) iff for each dimension $d$ (resp. each endolength $d$) there is only a finite number of distinct isoclasses of indecomposable finite-dimensional
$\Lambda$-modules $M$ with $\dim_k M = d$ (resp. with $\Endol{M} = d$).

In this article, we focus our attention on finite-dimensional algebras $\Lambda$  with infinite representation type. Since the ground field is perfect, this is equivalent 
to not being $e$-discrete, see \cite{ficlifts}. 
So, for such an algebra $\Lambda$ 
 there is at least one number $d\in \hueca{N}$ and an  infinite family  of non-isomorphic indecomposable finite-dimensional $\Lambda$-modules with endolength $d$.

\begin{definition}\label{D: repres embedding}
 Given two $k$-algebras $\Gamma$ and $\Lambda$,    
 a \emph{representation embedding from $\Gamma\g\mod$ to $\Lambda\g\mod$}  is a $k$-linear 
functor $H : \Gamma\g\mod \rightmap{} \Lambda\g\mod$ which  is exact, preserves indecomposability,  and reflects isomorphism classes.
\end{definition}

\begin{definition}\label{D: control endolength}
Given full subcategories ${\cal C}$ of $\Gamma\g\mod$ and ${\cal C'}$ of $\Lambda\g\mod$ we will say that a functor $H:{\cal C}\rightmap{}{\cal C}'$ \emph{controls endolength of indecomposables} iff there are some constants  $c,c'\in \hueca{N}$ such that 
$$\Endol{N}\leq c\times\Endol{H(N)}\hbox{  \ and \ }\Endol{H(N)}\leq c'\times \Endol{N},$$ for any indecomposable $N\in {\cal C}$.
\end{definition}

\begin{definition}\label{D: minimal algebra}
 A $k$-algebra $\Q$ is called \emph{minimal} iff it is of one of the following two types:
 \begin{enumerate}
  \item $\Q = T_{D_1 \times D_2} (V)$, where $D_1$ and $D_2$ are finite-dimensional division $k$-algebras and $V$ is a
simple $D_1\g D_2$-bimodule.
\item $\Q = T_D (V)$, where $D$ is a finite-dimensional division $k$-algebra and $V$ is a simple 
$D\g D$-bimodule.
 \end{enumerate} 
 The \emph{coefficient algebras} of such a minimal algebra $\Q$ are $D_1$ and $D_2$ in the first case, and $D$ in the second one. 
\end{definition}

The \emph{coefficient algebras} of a finite-dimensional algebra $\Lambda$   are, by definition, the opposites of the  endomorphism algebras of the simple $\Lambda$-modules. 
In this note we prove the following result.

\begin{theorem}\label{main thm for algs with minimal's} Let $\Lambda$ be a finite-dimensional $k$-algebra of infinite representation type. Then,  there is a functor  
which preserves indecomposables, reflects isomorphism classes, and controls endolength of indecomposables of one of the following two types:
\begin{enumerate}
 \item $H:\Q\g\mod_{\cal I}\rightmap{}\Lambda\g\mod$, where $\Q$ is a minimal algebra of the first type.
  \item $H:\Q\g\mod\rightmap{}\Lambda\g\mod$, where $\Q$ is a minimal algebra of the second type.
\end{enumerate}
In both cases, the algebra $\Q$ is of infinite representation type and 
with coefficient algebras some of the coefficient algebras of $\Lambda$. Moreover, the functor $H$ is of the form $Z\otimes_\Q-$, for some $\Lambda\g \Q$-bimodule $Z$ which is finitely generated by the right. 
\end{theorem}

From this, we derive:

\begin{theorem}\label{main thm for algs with PID's} Let $\Lambda$ be a finite-dimensional $k$-algebra of infinite representation type. Then, there is a  bounded principal ideal domain $\Gamma$ which is not $e$-discrete and a functor  
 $$G:\Gamma\g\mod\rightmap{}\Lambda\g\mod,$$ 
 which preserves indecomposables, reflects isomorphism classes, and controls endolength of indecomposables.  The functor $G$ is of the form $Z\otimes_\Gamma-$, for some $\Lambda\g \Gamma$-bimodule $Z$ which is finitely generated by the right. 
\end{theorem}

Then, we will show:

\begin{theorem}\label{main exact thm for algs with PID's} Let $\Lambda$ be a finite-dimensional $k$-algebra of infinite representation type. Then, there is a  bounded principal ideal domain $\Gamma$ which is not $e$-discrete and a
representation embedding  which controls endolength of indecomposables  
 $$G:\Gamma\g\mod\rightmap{}\Lambda\g\mod.$$ 
  Moreover, the functor $G$ is of the form $Z\otimes_\Gamma-$, for some $\Lambda\g \Gamma$-bimodule $Z$ which is free of finite rank by the right. 
\end{theorem}

As an application of the preceding theorem, we will prove the following variation of ``the second Brauer-Thrall conjecture'' involving  endolength. 

\begin{definition}\label{D: def de EBTII} We  say that a $k$-algebra $\Sigma$ \emph{satisfies EBTII} if there is an infinite sequence of natural numbers $d_1<d_2<\cdots$ such that for each one of these numbers $d_i$ there is an infinite family $\{M_j\}_j$ of non-isomorphic finite-dimensional indecomposable $\Sigma$-modules with $\Endol{M_j}=d_i$. 
\end{definition}
 
 \begin{theorem}\label{T: EBTII} 
 Any finite-dimensional algebra $\Lambda$ of infinite representation type  satisfies EBTII.  
 \end{theorem} 
 
 This result is a consequence of the following.  From (\ref{main exact thm for algs with PID's}), we have functors $G:\Gamma\g\mod\rightmap{}\Lambda\g\mod$  which  
 preserve indecomposables, reflect isomorphism classes, and control  endolength of indecomposables, they preserve the property EBTII, by  (\ref{L: EBTII se traslada con funtores que controlan endolength}). These algebras $\Gamma$ satisfy EBTII, as we shall verify in \S\ref{S: EBTII}.   

In view of the previous theorem, it is very natural to ask whether 
EBTII holds for a wider class of representation-infinite artin algebras. The problem is open even in the case of finite-dimensional algebras over a non-perfect fields.   

The proof of (\ref{main thm for algs with minimal's}) relies on the  theory of differential tensor algebras
(ditalgebras) and reduction functors first developed by the Kiev School of representation theory, see \cite{R-K} and \cite{D} (see also \cite{CB1}). For the general
background on ditalgebras and their module categories, we refer the reader to \cite{bsz}. The proofs of the results presented here are largely based on our previous work \cite{bps}. 

As in \cite{CB2}, we  study the phenomenon of existence of an infinite family of indecomposables in $\Lambda\g\mod$ with fixed endolength by reducing with matrix problems methods to the case of hereditary algebras. In \cite{CB2}, Crawley-Boevey uses the more 
general techniques of lift pairs which apply to general artin algebras.
Here, as in \cite{bps}, we restrict ourselves to the case of finite-dimensional algebras over perfect fields and use bocses techniques which permit us to exploit the properties of reduction functors. 

Representation embeddings have been of  great interest in representation theory of algebras, see for instance \cite{Sheila}, \cite{Bo1}, \cite{CB3}, or \cite{S}. Suffice it to mention that the definition of wildness requires the existence of such an embedding from $k\langle x,y\rangle\g\mod$. This work was inspired by Bongartz article \cite{Bo1}, which, under the assumption  that the ground field is algebraically closed, shows that for any 
finite-dimensional algebra $\Lambda$ of infinite representation type, there is a representation embedding $k[x]\g\mod\rightmap{}\Lambda\g\mod$. Here we show that some version of this holds in the more general situation of perfect fields. 

\section{Strategy for the proof of theorem  (\ref{main thm for algs with minimal's})}

We recall from \cite{bsz} and \cite{bps}, some basic definitions.

\begin{definition} Let ${\cal A}$ be a layered ditalgebra, with layer $(R, W )$,  see \cite[\S4]{bsz}. Given
$M \in  {\cal A}\g \Mod$, denote by $E_M := \End_{\cal A}(M)^{op}$ its endomorphism algebra. Then, $M$ admits a
structure of $R\g E_M$-bimodule, where $m \cdot (f^0, f^1 ) = f^0(m)$, for $m \in M$ and $( f^0 , f^1 ) \in E_M$. By
definition, \emph{the endolength of $M$}, denoted by $\Endol{M}$, is the length of $M$ as a right $E_M$-module. 

As in the case of algebras, the ditalgebra ${\cal A}$ is called \emph{$e$-discrete} iff, for each endolength $d$, there are only a finite number of distinct isoclasses of finite-dimensional indecomposable ${\cal A}$-modules with endolength $d$.  
\end{definition}

\begin{definition}\label{D: admissible ditalgebra}  
Let ${\cal A} = (T, \delta)$ be a triangular ditalgebra, with layer 
$(R, W )$, over the field $k$.
Then, ${\cal A}$  is called \emph{admissible} iff $R \cong 
 D_1 \times \cdots \times D_n$, for some finite-dimensional division $k$-algebras
$D_1,\ldots,D_n$, and $W$ is finitely generated as an $R\g R$-bimodule. \emph{The coefficient algebras of ${\cal A}$} are the algebras $D_1,\dots,D_n$. 
\end{definition}

Since our ground field $k$ is perfect, any finite-dimensional $k$-algebra $\Lambda$ splits over its radical,
thus $\Lambda = S \oplus J$, where $J$ is the radical of $\Lambda$. Then, we can consider the Drozd's ditalgebra 
${\cal D}$ of $\Lambda$, see \cite[(19.1)]{bsz}, which is an admissible ditalgebra if and only if $\Lambda$ is basic. In this case, the coefficient algebras of ${\cal D}$ coincide with those of $\Lambda$. 

For the proof of our theorem (\ref{main thm for algs with minimal's}), we can assume that $\Lambda$ is basic, because Morita equivalences are representation embeddings which preserve coefficient algebras and control endolength of indecomposables.  The composition of representations embeddings is a representation embedding, see \cite{Bo1}.     

\begin{remark}\label{R: Lambda tri implies D no es e-dscreto}
If $\Lambda$ has infinite representation type, as we mentioned before, it is not $e$-discrete, so there is an fixed endolength $d$ and an  infinite family $\{M_i\}_{i\in I}$ of finite-dimensional non-isomorphic indecomposable $\Lambda$-modules with $\Endol{M_i}=d$. 

Recall that we have the cokernel functor $\Cok:{\cal P}(\Lambda)\rightmap{}\Lambda\g\mod$ and the equivalence $\Xi_\Lambda:{\cal D}\g\mod\rightmap{}{\cal P}^1(\Lambda)$, studied for instance in \cite[\S18]{bsz}, whose composition functor permits to transfer many properties from the category ${\cal D}\g\mod$ to $\Lambda\g\mod$.  Then, for instance by \cite[(18.10)(4)]{bsz}, there is a family of non-isomorphic indecomposable objects $\{X_i\}_{i\in I}$ in ${\cal P}^2(\Lambda)$ such that $\Cok(X_i)\cong M_i$, for all $i\in I$.  So, for each $i\in I$, we can  choose $N_i\in {\cal D}\g\mod$ with $\Xi_\Lambda(N_i)\cong X_i$.  Thus, $\Cok\Xi_\Lambda(N_i)\cong M_i$ and, 
from \cite[(4.4)]{bps}, there is a natural number $d_0$ and an infinite subset $I_0\subseteq I$ such that  $\{N_i\}_{i\in I_0}$ is an infinite family of finite-dimensional non-isomorphic indecomposable ${\cal D}$-modules with $\Endol{N_i}=d_0$.  So we have that ${\cal D}$ is not $e$-discrete. 
\end{remark}

The strategy of the proof of our theorem (\ref{main thm for algs with minimal's}) includes first the proof of the following statement for admissible ditalgebras. 

\begin{theorem}\label{main thm for ditalgs} 
Let ${\cal A}$ be an admissible $k$-ditalgebra which  is not $e$-discrete. Then,  there is a minimal algebra $\Q$ of infinite representation type and a functor  $F:\Q\g\mod\rightmap{}{\cal A}\g\mod$ which preserves indecomposables and reflects isomorphism classes. Moreover, the coefficient algebras of $\Q$ are some of those of ${\cal A}$. 
\end{theorem}

This theorem will be proved by a reduction procedure.  Then, we apply it to the Drozd's ditalgebra ${\cal D}$ of $\Lambda$ and consider the composition of the  functor $F$ provided by (\ref{main thm for ditalgs}) with    $\Cok\Xi_\Lambda:{\cal D}\g\mod\rightmap{}\Lambda\g\mod$ to obtain the functor we are looking for.  

The proof of the last theorem will be done by induction using the notion of $e$-norm of indecomposable finite-dimensional modules  discussed in the following.

\section{Norms  and reduction functors}

Suppose that  $(R, W)$ is the triangular layer of an admissible ditalgebra ${\cal A}$.  Assume that $1=\sum_{i=1}^ne_i$ is the  decomposition of the unit of $R$ as a sum of central primitive orthogonal idempotents. Thus, with the notation of (\ref{D: admissible ditalgebra}), we have the finite-dimensional division $k$-algebra $D_i=e_iD_ie_i$, for $i\in [1,n]$. 

In the discussion of \cite{bps} the following notion of norm is relevant. 

\begin{definition}\label{D: norm} 
Assume that ${\cal A}$ is an admissible $k$-ditalgebra with layer $(R, W )$. Then,
for $M \in  {\cal A}\g\mod$, its \emph{norm} is the number
$$\vert\vert M\vert\vert = \dim_k \Hom_R (W_0 \otimes_R M, M).$$

 Then, the 
 \emph{length of $M$} is $ \ell_R (M) =
\sum_{i=1}^n \ell_{D_i} (e_i M)$. The \emph{support of $M$} is the set of
idempotents $e_i$ with $e_i M \not= 0$. The ${\cal A}$-module
$M$ 
is called \emph{sincere} if $e_i M\not=  0$, for all $i \in [1, n]$.

For $M \in {\cal A}\g\mod$, we have $$\vert\vert M\vert\vert = \sum_{i, j}\dim_k (e_i W_0 e_j ) \ell_{D_i} (e_i M)\ell_{D_j} (e_j M) .$$
\end{definition}

In this note, when dealing with finite-dimensional indecomposable ${\cal A}$-modules with finite endolength, we will use the following modified norm.  

\begin{definition}\label{D: e-norm}
Let $M$ be a $k$-finite-dimensional indecomposable ${\cal A}$-module and denote by $E_M= \End_{\cal A}(M)^{op}$. Then,  the right $E_M$-module $M$
has the direct sum decomposition $M = \bigoplus_{i=1}^n e_i M$, of right $E_M$-modules.  We get  
$$
\Endol{M}=\sum_{i=1}^n\ell_{E_M}(e_iM). 
$$
From \cite[(5.6)]{bsz}, ${\cal A}$ is a Roiter ditalgebra and, from \cite[(5.12)]{bsz}, we know that $E_M$ is a local finite-dimensional $k$-algebra. 
Since the field $k$ is perfect, there is a division $k$-subalgebra $D_M$ of $E_M$ and a $D_M\g D_M$-bimodule decomposition 
$E_M = D_M \bigoplus \rad E_M$.

So every finite-length right $E_M$-module $N$  is a finite-dimensional right $D_M$-vector space and we have $\ell_{E_M}(N)=\ell_{D_M}(N)$. So,  
$\Endol{M} = \ell_{D_M}(M)$  and $\ell_{E_M}(e_iM)=\ell_{D_M}(e_iM)$,  for  $i\in [1,n]$. Hence, 
$$
\Endol{M} = \sum_{i=1}^n
\ell_{D_M}(e_i M).$$

Let $c=c_{\cal A}$ be the least common multiple of $d_1,d_2,\ldots ,d_n$, where  $d_i = \dim_k D_i$. So, for each $i$, there is a natural number $c_i$ with  $c = c_i d_i$. Then, define the \emph{$e$-norm} of $M$  by 
$$\vert\vert M\vert\vert_e:=\sum_{i,j}c_ic_j\dim_k(e_i W_0 e_j )
\ell_{D_M}(e_iM) \ell_{D_M}(e_jM).$$
\end{definition}

\begin{remark}\label{R: not e-discrete -> tiene infinite q-family}
From this definition, we get that $\vert\vert M\vert\vert_e\leq c^2\dim_k(W_0)\Endol{M}^2$. So if there is an infinite family of non-isomorphic finite-dimensional indecomposable ${\cal A}$-modules with endolength bounded by $q$, then their $e$-norms  are bounded by $c^2\dim_k(W_0)q^2$.

If ${\cal A}$ is any admissible ditalgebra and $q\in \hueca{N}$, we say that ${\cal A}$ is of \emph{infinite $q$-type} iff ${\cal A}$ admits an infinite family of non-isomorphic finite-dimensional indecomposable  ${\cal A}$-modules $M$ with $e$-norm $\vert\vert M\vert\vert_e=q$.  

So, if ${\cal A}$ is not $e$-discrete, it is of infinite $q$-type for some $q\in \hueca{N}$. 
\end{remark}

\begin{lemma} Keep the notations of (\ref{D: norm}) and (\ref{D: e-norm}), for a  given finite-dimensional  indecomposable ${\cal A}$-module $M$. Moreover, make $d_M=\dim_kD_M$, then we have:
$$\vert\vert M\vert\vert_e=\frac{c_{\cal A}^2}{d_M^2}\vert\vert M\vert\vert.$$
\end{lemma}

\begin{proof}
Recall that  $\vert\vert M\vert\vert=
\sum_{i,j}
\dim_k (e_i W_0 e_j )\ell_{D_i}(e_iM) \ell_{D_j} (e_jM)$ and notice that, for 
 each  $i\in [1,n]$, we have 
$$c\ell_{D_i}(e_iM)=c_id_i\ell_{D_i}(e_iM)=c_i\dim_ke_iM=c_id_M\ell_{D_M}(e_iM).$$
\end{proof}

In the next statement we describe the behavior of the reduction functors towards the $e$-norm of finite-dimensional indecomposables. 
We say that an algebra $B$ is \emph{an initial subalgebra of an admissible ditalgebra} ${\cal A}$, with layer $(R,W)$,  if $B=T_R(W^1_0)$, where $W_0^1$ is the smaller term in the triangular filtration of $W_0$, see the reminder \cite[(2.3)]{bps}.

\begin{lemma}\label{L: reduction functors and enorm}
 Assume that ${\cal A}$ is an admissible $k$-ditalgebra with layer $(R,W)$. Suppose that $z\in \{d,r,X\}$ and ${\cal A}^z$ is the admissible ditalgebra obtained from ${\cal A}$ by one of the basic reductions:  deletion of an idempotent (case $z=d$, as in \cite[(8.17)]{bsz}), regularization (case $z=r$, as in \cite[(8.19)]{bsz}), or reduction by a $B$-module $X$, where $B$ is an initial subalgebra of ${\cal A}$ and $X$ is a finite direct sum of pairwise non-isomorphic finite-dimensional indecomposable $B$-modules $X_1,\ldots,X_t$ with $\End_B(X_i)^{op}/\rad\End_B(X_i)^{op}$ isomorphic to some of the coefficient algebras of ${\cal A}$ (case $z=X$, a particular case of \cite[(7.3)]{bps}). Consider also the associated full and faithful functor 
 $F^z:{\cal A}^z\g\mod\rightmap{}{\cal A}\g\mod$. Then,  the following holds for any finite-dimensional indecomposable $N\in {\cal A}^z\g\mod$.
 \begin{enumerate}
  \item $\vert\vert N\vert\vert_e\leq \vert\vert  F^d(N)\vert\vert_e$.
  \item $\vert\vert N\vert\vert_e < \vert\vert F^r(N)\vert\vert_e$, whenever $F^r(N)$ is sincere.
\item $\vert\vert N\vert\vert_e < \vert\vert F^X(N)\vert\vert_e$, whenever $F^X(N)$ is sincere and $W'_0\not=0$.
  \end{enumerate}
\end{lemma}

\begin{proof} Recall that the layer $(R,W)$ of the admissible ditalgebra ${\cal A}$ is such that  $R=D_1\times\cdots\times D_n$ as in (\ref{D: admissible ditalgebra}). 

In order to show 1, for simplicity, assume that the idempotent to be deleted is $e_1$, then the layer $(R^d,W^d)$ of  ${\cal A}^d$ is such that 
$R^d=D_2\times\cdots \times D_n$, so clearly $c_{{\cal A}^d}\leq c_{\cal A}$ and, since $F^d$ is full and faithful, we get $d_N=d_{F^d(N)}$. Moreover, since $\vert\vert N\vert\vert=\vert\vert F^d(N)\vert\vert$, we have 
$$\vert\vert N\vert\vert_e=
\frac{c^2_{{\cal A}^d}}{d^2_{N}}\vert\vert N\vert\vert\leq
\frac{c^2_{\cal A}}{d^2_{F^d(N)}}\vert\vert F^d(N)\vert\vert=
\vert\vert F^d(N)\vert\vert_e.$$
 
 For 2, we notice that the layer $(R^r,W^r)$ of ${\cal A}^r$ is such that $R^r=R$, so $c_{{\cal A}^r}= c_{\cal A}$. Again, since $F^r$ is full and faithful, we get $d_N=d_{F^r(N)}$. Moreover, since $\vert\vert N\vert\vert<\vert\vert F^r(N)\vert\vert$, whenever  $F^r(N)$ is sincere, as before, we have 
$$\vert\vert N\vert\vert_e=
\frac{c^2_{{\cal A}^r}}{d^2_N}\vert\vert N\vert\vert< 
\frac{c^2_{\cal A}}{d^2_{F^r(N)}}\vert\vert F^r(N)\vert\vert=
\vert\vert F^r(N)\vert\vert_e.$$

In the last case, the assumption implies that $c_{{\cal A}^X}\leq c_{\cal A}$, and we can proceed as before, using that for $F^X(N)$ sincere we have $\vert\vert N\vert\vert<\vert\vert F^X(N)\vert\vert$, see  \cite[(7.3)(1)]{bps}.
\end{proof}

\section{Reduction to minimal algebras}

In this section, we prove the following theorem (\ref{T: reduccion a minimales}), which  is a stronger version of  (\ref{main thm for ditalgs}). 
Its statement and proof use the following  notions and preliminary arguments.  
As in the case of algebras, we say that a functor $F:{\cal A}'\g\mod\rightmap{}{\cal A}\g\mod$, where ${\cal A}'$ and ${\cal A}$ are layered ditalgebras,  \emph{controls endolength of indecomposables}   if there are some $c,c'\in \hueca{N}$ such that  
$\Endol{N}\leq c\times\Endol{F(N)}$ and $\Endol{F(N)}\leq c'\times \Endol{N}$, for all indecomposables $N\in {\cal A}'\g\mod$.

Recall that given any ditalgebra,  ${\cal A}=(T,\delta)$, there is a canonical embedding functor $L_{\cal A}:A\g\mod\rightmap{}{\cal A}\g\mod$, where $A$ is the subalgebra of $T$ formed   by the homogeneous elements of degree zero. It maps every $A$-module to itself and each morphism $f\in\Hom_A(M,N)$ is mapped onto $(f,0)\in\Hom_{\cal A}(M,N)$, see \cite[\S2]{bsz}. Clearly, the functor $L_{\cal A}$ reflects indecomposability.

\begin{definition}\label{D: quasimnimal} 
We say that an admissible  ditalgebra ${\cal Q}=(T,\delta)$ is \emph{quasi-minimal} iff $T$ has layer $(R,W)$ such that $R=D_1$ or $R=D_1\times D_2$, there is an $R$-$R$-bimodule decomposition $W_0=W'_0\oplus W''_0$ is such that
 $W'_0=e_1W'_0e_1$ in the first case and  $W'_0=e_2W_0'e_1$ in the second one, and  moreover,    in any case, $W'_0$ is a simple $R$-$R$-bimodule with $\delta(W'_0)=0$. 
 
 Any such quasi-minimal ditalgebra ${\cal Q}$ determines a minimal algebra $\Q:=T_R(W'_0)$, which we call \emph{the minimal algebra of ${\cal Q}$}. 
\end{definition}

\begin{remark}\label{R: sobre quasi-dits}
Quasi-minimal ditalgebras ${\cal Q}=(T,\delta)$ have the following important feature. If $Q$ denotes the subalgebra of $T$ formed by the homogeneous elements of degree 0, we have the  corresponding canonical embedding functor $L_{\cal Q}:Q\g\mod\rightmap{}{\cal Q}\g\mod$. Consider  the minimal algebra $\Q$ of ${\cal Q}$,   the projection morphism of algebras $\pi:Q=T_R(W)\rightmap{}T_R(W'_0)=\Q$ ,  and the corresponding \emph{extension functor} $E:\Q\g\Mod\rightmap{}Q\g\Mod$, which is given by scalar restriction through $\pi$, thus $E\cong E(\Q)\otimes_{\Q}-$. Then, since $\delta(W'_0)=0$ and $W''_0E(N)=0$,  for any finite-dimensional indecomposable $\Q$-module $N$, we have
 $$\hueca{E}:=\End_{{\cal Q}}(L_{\cal Q}E(N))^{op} =\hueca{E}_N^0 \oplus \hueca{E}_N^1,$$
where $\hueca{E}_N^0$ and  $\hueca{E}_N^1$ are the subspaces of $\hueca{E}$ formed, respectively,  by the endomorphisms of $E(N)$ of the form $(f^0,0)$, and by the endomorphisms of $E(N)$ of the form $(0,f^1)$. Moreover, $\hueca{E}_N^0 = L_{\cal Q} E(\End_{\Q}(N)^{op})$. 

Indeed, given any $g^0\in \End_{\Q}(N)^{op}$, from the equality  $W''_0E(N)=0$, we have $g^0(wm)=0=wg^0(m)$, for all $w\in W''_0$ and $m\in E(N)$.  Since $Q$ is generated by $R$, $W'_0$, and $W''_0$, we have that $g^0\in \End_Q(E(N))^{op}$, so $L_{\cal Q}E(g^0)=(g^0,0)\in \hueca{E}_N^0$. Moreover, 
given $f=(f^0,f^1)\in \hueca{E}$, the equality $\delta(W'_0)=0$ implies that $f^0\in \End_{\Q}(N)^{op}$.   Thus, $(f^0,0)\in \hueca{E}^0_N$,  $(0,f^1)=(f^0,f^1)-(f^0,0)\in \hueca{E}$, and $\hueca{E}_N^0 = L_{\cal Q} E(\End_{\Q}(N)^{op})$.  

 Our ditalgebra ${\cal Q}$ is triangular, so, from \cite[(5.4)]{bsz}, $\hueca{E}^1_N$ is a nilpotent ideal of $\hueca{E}$, so we know that $\hueca{E}_N^1 \subseteq  \Rad \hueca{E}$.
Since $k$ is perfect, we also have 
$\End_{\Q}(N)^{op} = K_N \oplus \Rad(\End_{\Q}(N)^{op})$, 
where $K_N$ is a division algebra. From the preceding facts,  we obtain the  equality
$$\End_{{\cal Q}}(L_{\cal Q}E(N))^{op} 
=
L_{\cal Q}E(K_N)\oplus 
\Rad(\End_{{\cal Q}}(L_{\cal Q}E(N))^{op}).$$  
Thus, the functor $\LQ_\Q:=L_{\cal Q}E:\Q\g\mod\rightmap{}{\cal Q}\g\mod$ preserves indecomposables. It reflects isomorphism classes because, ${\cal Q}$ is a Roiter ditalgebra, see \cite[(5.8) and (5.6)]{bsz}. We will keep the preceding notation whenever we deal with quasi-minimal ditalgebras.  
\end{remark}

We have the following elementary statement relating dimension and endolength.

\begin{lemma}\label{L: dim vs endol}
Let $\Sigma$ be any $k$-algebra (not necessarily finite-dimensional) and $M$ an indecomposable finite-dimensional 
$\Sigma$-module. Consider the local algebra $E:=\End_{\Sigma}(M)^{op}$ and the division algebra $K_M:=E/\Rad E$. Then, we have 
$$\Endol{M} = \dim_kM/\dim_kK_M.$$
\end{lemma}

\begin{proof} Since $M$ is finite-dimensional, its length as an $E$-module is finite. Consider a composition series $0=M_0 \subset M_1  \subset \cdots \subset M_l =M$ 
for the right $E$-module $M$. Then, each simple $E$-module $M_i/M_{i-1}$ is a simple $K_M$-module: a one-dimensional $K_M$-vector space. So, we have $ \dim_k(M_i/M_{i-1}) = \dim_kK_M$, for each $i$. Thus, 
$\dim_kM = l\dim_kK_M$. Our statement follows from this. 
\end{proof}

\begin{remark}\label{R: Phi preserves endolength of indecs} Similarly, given any admissible ditalgebra ${\cal A}$ and an indecomposable $M\in {\cal A}\g\mod$, if we consider the local algebra $E:=\End_{\cal A}(M)^{op}$ and the division algebra $K_M:=E/\Rad E$, we have 
$\Endol{M} = \dim_kM/\dim_kK_M$. 

Given a quasi-minimal ditalgebra ${\cal Q}$, with the notation of (\ref{R: sobre quasi-dits}), given an indecomposable finite-dimensional $\Q$-module $M$, 
since the functor $\LQ_{\Q}$ preserves dimension and $K_M\cong \LQ_{\Q}(K_M)\cong K_{\LQ_{\Q}(M)}$, we obtain from (\ref{R: sobre quasi-dits}) and (\ref{L: dim vs endol})  that $\Endol{\LQ_{\Q}(M)} = 
\dim_k\LQ_{\Q}(M)/
\dim_kK_{\LQ_{\Q}(M)}=
\dim_kM/\dim_kK_M=\Endol{M}$. Thus, $\LQ_{\Q}:\Q\g\mod\rightmap{}{\cal Q}\g\mod$ preserves endolength of indecomposable modules.
\end{remark}

\begin{definition}\label{D: control de dim} Let ${\cal C}$ and ${\cal C'}$ be full subcategories of finite-dimensional modules over some algebras or some ditalgebras. Then, we say that a functor $F:{\cal C}\rightmap{}{\cal C}'$ \emph{controls dimension} iff  there are constants $c,c'\in \hueca{N}$ such that
$$   \dim_kM\leq c\times\dim_kFM \hbox{ \ and \ } \dim_kFM\leq c'\times\dim_kM,$$
for all  $M\in {\cal C}$.
\end{definition}

\begin{remark}\label{R: controla dim -> controla endol} From (\ref{L: dim vs endol}), we know that 
any functor $F:{\cal C}\rightmap{}{\cal C}'$ which controls dimension and induces isomorphisms  $K_M\cong K_{FM}$, for all indecomposable $M$, controls endolength of indecomposables. 
\end{remark}

\begin{lemma}\label{L: F^X controls endolength}
Let ${\cal A}$ be an admissible  ditalgebra. Assume that ${\cal A}^X$ is the layered ditalgebra obtained from ${\cal A}$ by reduction, using the $B$-module $X$, where $B$ is an initial subalgebra of ${\cal A}$ and $X$ is a finite direct sum of non-isomorphic finite-dimensional indecomposable $B$-modules. Assume that $X$ is a complete admissible triangular $B$-module and let   $F^X : {\cal A}^X \g\Mod \rightmap{} {\cal A}\g\Mod$ be the associated functor, as in \cite[12.10]{bsz}. Then, $F^X$ controls dimension and  controls endolength of indecomposables. 
\end{lemma}

\begin{proof} The finite-dimensional algebra $\Gamma = \End_B(X)^{op}$ admits the splitting $\Gamma = S \oplus P$, where $P$ is the radical of $\Gamma$, because the ground field is perfect.  The semisimple algebra $S$ is basic because the indecomposable direct summands of $X$ are pairwise non-isomorphic. Let us denote by $f_1,\ldots,f_t$ the orthogonal primitive central idempotents given by the decomposition of the unit of $S$.

From \cite[(2.7)]{bps}, we already know that or all $N \in {\cal A}^X \g\Mod$, we have that
$\Endol{F^X(N)}\leq \dim_k X\times \Endol{N}$. 

It is clear that $\dim_kF^X(N)\leq \dim_kX\times \dim_kN$. 
Moreover, we have 
$$\dim_kF^X(N)=\dim_kX\otimes_SN=\sum_i\dim_k
Xf_i\otimes_{Sf_i}f_iN\geq \sum_i\dim_kf_iN=\dim_kN,$$
because $\dim_kXf_i\otimes_{Sf_i}f_iM=\dim_k(\dim_{Sf_i}Xf_i)f_iM\geq \dim_k f_iN,$ with $Xf_i\not=0$. So, the functor  $F^X$ controls dimension, and we can apply (\ref{R: controla dim -> controla endol}).  
\end{proof}

\begin{theorem}\label{T: reduccion a minimales} 
For any admissible $k$-ditalgebra ${\cal A}$ not $e$-discrete,   
there is a quasi-minimal ditalgebra ${\cal Q}$ with minimal algebra $\Q$ of infinite representation type,  and a faithful functor $F:\Q\g\mod\rightmap{}{\cal A}\g\mod$ which preserves indecomposability,   reflects isomorphism classes, and controls endolength of indecomposables.
 The functor $F$ is the composition $$\Q\g\mod\rightmap{\LQ_{\Q}}{\cal Q}\g\mod\rightmap{G}{\cal A}\g\mod,$$ 
where $G$ is a composition of  full and faithful reduction functors associated to a finite sequence of reductions which transform ${\cal A}$ into ${\cal Q}$.  Moreover, the coefficient algebras of $\Q$ are some of those of ${\cal A}$. 
\end{theorem}

\begin{proof} Assume that $R =D_1\times\cdots\times D_n$ and $1=\sum_{i=1}^ne_i$ is a decomposition of the unit of $R$ as a sum of central ortogonal idempotents, so each $e_i Re_i = D_i$ is a  
finite-dimensional $k$-algebra, say of dimension $d_i= \dim_kD_i$.
We will use the $e$-norm $\vert\vert M\vert\vert_e$ defined for indecomposable modules in  ${\cal A}\g\mod$. 

By definition, given $q\in \hueca{N}$, \emph{a $q$-infinite family of ${\cal A}\g\mod$} is an infinite family ${\cal F}$ of non-isomorphic finite-dimensional indecomposable ${\cal A}$-modules $M$ with common  $e$-norm $\vert\vert M\vert\vert_e=q$.  
Since ${\cal A}$ is not $e$-discrete, 
${\cal A}\g\mod$ admits a $q$-infinite family ${\cal F}$.

If there are infinitely many ${\cal A}$-modules $M$ in the family ${\cal F}$, with $e_iM=0$, for a fixed $i\in [1,n]$, we consider the ditalgebra ${\cal A}^{d_1}$ obtained from ${\cal A}$ by deletion of the idempotent $e_i$ and the associated full and faithful functor $F^{d_1}:{\cal A}^{d_1}\g\mod\rightmap{}{\cal A}\g\mod$. Then, ${\cal A}^{d_1}$ is an admissible ditalgebra and  we know that  $1\leq  \vert\vert N\vert\vert_e\leq \vert\vert F^{d_1}(N)\vert\vert_e=q$, whenever $M\in {\cal F}$  and $F^{d_1}(N)\cong M$.  So there is a $q_1$-infinite family ${\cal F}^{d_1}$ in ${\cal A}^{d_1}\g\mod$ with $q_1\leq q$. In a finite number of such deletion of idempotents, we obtain an admissible ditalgebra ${\cal A}'={\cal A}^{d_1\cdots d_m}$  and a $q'$-infinite family ${\cal F}'$ of sincere modules in 
${\cal A}'\g\mod$ with $q'\leq q$.

If there is a quasi-minimal ditalgebra ${\cal Q}$ and a functor $F:\Q\g\mod\rightmap{}{\cal A}'\g\mod$ as in the statement of the theorem for the ditalgebra ${\cal A}'$, then the composition $F^{d_1}\cdots F^{d_m}F:\Q\g\mod\rightmap{}{\cal A}\g\mod$ is the required functor for ${\cal A}$, because each functor $F^{d_i}$ 
 preserves endolength, see \cite[(2.5)]{bps}. 

 So we can assume that we already started with a $q$-infinite family of sincere  
${\cal A}$-modules  ${\cal F}$ and proceed to the proof of the theorem by induction on $q$.

In case $q=1$, for each $M\in {\cal F}$, with the notation of (\ref{D: e-norm}), we have 
$$
1 = \vert\vert M\vert\vert_e =\sum_{i,j}
c_ic_j \dim_k (e_i W_0 e_j) \ell_{D_M}(e_i M )\ell_{D_M}(e_jM).$$
This implies that $ n = 1$ or   $n = 2$. If $n=1$, 
$R= D_1$ and $ \dim_k (e_1 W_0 e_1 ) = 1$. Moreover, 
 $D_1 = k$  and $A= T_k (W_0 )$ with  $\dim_k W_0 = 1$. Here $\delta(W_0)\subseteq W_1$, because the triangular filtration of $W_0$ is just $0\subseteq W_0$.  
If $\delta(W_0)\not=0$, since $W^1_0$ is a simple $R\g R$-bimodule and $W_0^1\cap \Ker \delta=0$, we can consider the admissible ditalgebra ${\cal A}^r$ obtained from ${\cal A}$ by regularization and the associated equivalence functor $ F^r : {\cal A}^r \g\mod \rightmap{} {\cal A}\g\mod$, see \cite[(8.19)]{bsz}. But, here, ${\cal A}^r$ has only one indecomposable module, so ${\cal A}$ would not be of infinite representation type.  So, $\delta(W_0)=0$, ${\cal A}$ is quasi-minimal,  
 and we can take ${\cal Q}={\cal A}$, $\Q=A\cong k[x]$, and $G=Id$. The functor $F=\LQ_{\Q}$ preserves indecomposables and reflects isomorphisms by (\ref{R: sobre quasi-dits}). It preserves endolength of indecomposables by (\ref{R: Phi preserves endolength of indecs}).

In case $n=2$, we have $ R = D_1 \times D_2$ and, since, $\dim_k W_0 = 1$, we get  $D_1 = D_2 = k$. In this case, since $L_{\cal A}:A\g\mod\rightmap{}{\cal A}\g\mod$ reflects indecomposability and isomorphism classes, and $A=T_{k\times k}(W_0)$ is of finite representation type, 
then ${\cal A}$ is also of finite representation type, which is not the case. Hence, this case does not occur. 
So the base of the induction is established.

Now, assume $q\geq 2$ and that the statement of the theorem holds for every  admissible ditalgebra ${\cal A}'$ with a  $q'$-infinite family ${\cal F}'$ in ${\cal A}'\g\mod$ with  $q'<q$. 
We will proceed as in the proof of \cite[(7.5)]{bps}. 

Since $k$ is perfect, the algebra $R\otimes_kR^{op}$ is semisimple and we can look at the additive triangular filtration $0=W^0_0\subseteq W_0^1\subseteq\cdots \subseteq W_0^s=W_0$, see \cite[(5.1)]{bsz}.  We can assume that $W_0^1$ is a simple direct summand of the $R\g R$-bimodule $W_0$. The triangularity conditions imply that $\delta(W_0^1)\subseteq W_1$. Consider $R\g R$-bimodule decompositions $W_0=W_0^1\oplus W''_0$ and $W_1=\delta(W^1_0)\oplus W''_1$. We have the following two possibilities.

\medskip
\noindent\emph{Case 1: $\delta(W_0^1 )\not= 0$.}
\medskip

Since $W^1_0$ is a simple $R\g R$-bimodule and $W_0^1\cap \Ker \delta=0$, we can consider the admissible ditalgebra ${\cal A}^r$ obtained from ${\cal A}$ by regularization and the associated equivalence functor $ F^r : {\cal A}^r \g\mod \rightmap{} {\cal A}\g\mod$, see \cite[(8.19)]{bsz}. For any  indecomposable $N\in {\cal A}^r\g\mod$ with $F^r(N)\cong M\in {\cal F}$, we have 
$\vert\vert N\vert\vert_e<\vert\vert M\vert\vert_e=q$, by (\ref{L: reduction functors and enorm})(2). Then, there is a  family ${\cal F}^r$ of such ${\cal A}^r$-modules  $N$ which is a $q'$-infinite family in ${\cal A}^r\g\mod$  with $q'<q$. So,   we can apply the induction hypothesis to ${\cal A}^r$ 
to derive the existence of a quasi-minimal  ditalgebra ${\cal Q}$ with minimal algebra  $\Q$ and a   functor $F:\Q\g\mod\rightmap{}{\cal A}^r\g\mod$ as in the statement of the theorem, so the composition $F^rF:\Q\g\mod\rightmap{}{\cal A}\g\mod$ is the wanted functor for ${\cal A}$. Here, the functor $F^r$ preserves endolength, see \cite[(2.6)]{bps}. 

\medskip
\noindent\emph{Case 2: $\delta(W_0^1 ) = 0$}.
\medskip

Since $W_0^1$ is a simple $R\g R$-bimodule, we get $W_0^1=e_jW_0^1e_i$, for some $i,j\in [1,n]$, and we have two subcases.

\medskip

\emph{Case 2.a:} If $i=j$, for simplicity, we can assume that $i=1$. 
Then, we get $T_R (W_0^1 ) = T_{D_1} (W_0^1) \times D_2 \times \cdots \times D_n$. Consider the admissible ditalgebra ${\cal A}^d$ obtained from ${\cal A}$ by deletion of the idempotent $e=1-e_1$ and the corresponding full and faithful functor $F^d:{\cal A}^d\g\Mod\rightmap{}{\cal A}\g\Mod$. Then, ${\cal Q}:={\cal A}^d$ is a quasi-minimal ditalgebra, with minimal algebra $\Q=T_{D_1}(W_0^1)$ of infinite representation type, and  the composition  of functors 
$$
\Q\g\mod \rightmap{ \LQ_{\Q} } {\cal Q}\g\mod\rightmap{F^d}{\cal A}\g\mod$$
is the required functor $F$ for ${\cal A}$, in this case.
\medskip

\emph{Case 2.b:} If $i\not= j$, for simplicity, we assume that $i=1$ and $j=2$, and we get 
$B:=T_R (W_0^1 ) = T_{D_1\times D_2}(W_0^1)  \times D_3\times \cdots\times  D_n$. 

If $B$ is of infinite representation type, so is the algebra $ T_{D_1\times D_2}(W_0^1)$. Then, we can consider the  ditalgebra ${\cal A}^d$ obtained from ${\cal A}$ by deletion of the idempotent $e=1-e_1-e_2$ and the associated  full and faithful functor $F^d:{\cal A}^d\g\Mod\rightmap{}{\cal A}\g\Mod$. Then, ${\cal Q}:={\cal A}^d$ is a quasi-minimal ditalgebra, with minimal algebra $\Q=T_{D_1\times D_2}(W_0^1)$ of infinite representation type, and    the composition of functors 
$$
\Q\g\mod \rightmap{ \LQ_{\Q} } {\cal Q}\g\mod\rightmap{F^d}{\cal A}\g\mod$$
is the required functor $F$ for ${\cal A}$.

On the other hand, if $B$ is of finite representation type, so is the algebra   $B_0:=T_{D_1\times D_2}(W_0^1)$. Let $X_1,\ldots, X_l$ be a complete set of representatives of the isomorphism classes of the indecomposables in   
$B_0\g \mod$. Since $B_0$ is hereditary of finite representation type, then the division algebras $D_{X_i}=\End_{B_0}(X_i)^{op}/\rad\End_{B_0}(X_i)^{op}$ coincide with some $D_1$ or $D_2$ (this follows for instance from \cite[VIII.1: (1.5) and (1.13)]{ars}, see also \cite{CB2}). 

Consider the $B$-module $X=X_1\oplus\cdots\oplus X_l\oplus eR$ and  the admissible ditalgebra ${\cal A}^X$ obtained from ${\cal A}$ by reduction by the module $X$ and the associated full and faithful functor $F^X:{\cal A}^X\g\mod\rightmap{}{\cal A}\g\mod$, as in (\ref{L: reduction functors and enorm}). 
The layer of ${\cal A}^X$ has the form 
$(R^X, W^X)$ with $R^X = D_{X_1}\times\cdots\times D_{X_l}\times D_3\times\cdots \times D_n$. From (\ref{L: reduction functors and enorm}), whenever $F^X(N)\cong M\in {\cal F}$, we have 
$\vert\vert N \vert\vert_e <\vert\vert M\vert\vert_e=q$. But from 
\cite[(7.3)(2)]{bps}, we know that $F^X$ is dense. So, there in an infinite family in  ${\cal F}^X = \{N \mid F^X (N ) \in {\cal F}\}$,  which shows that the admissible ditalgebra ${\cal A}^X$ is of $q'$-infinite type for some $q'<q$.  Then, we can apply the induction hypothesis to ${\cal A}^X$ and obtain a quasi-minimal ditalgebra ${\cal Q}$ and a functor $F:\Q\g\mod\rightmap{}{\cal A}^X\g\mod$ as in the statement of the theorem, so the composition $F^XF:\Q\g\mod\rightmap{}{\cal A}\g\mod$ gives the wanted functor for ${\cal A}$. Here, the functor $F^X$ controls endolength of indecomposables by  (\ref{T: reduccion a minimales}). 

We have constructed, in each case, a  functor $F:\Q\g\mod\rightmap{}{\cal A}\g\mod$ which preserves indecomposability,  reflects isomorphism classes, and controls endolength of indecomposables.  
\end{proof}

The following lemma is very useful. It can be applied,  
to  Morita equivalences.  

\begin{lemma}\label{L: endols y repembed fieles y plenos} Given two algebras $\Gamma$ and $Q$,  
assume that $\Psi:\Gamma\g\mod\rightmap{}Q\g\mod$ is a full and faithful  functor of the form $P\otimes_\Gamma-$, for some  $Q\g\Gamma$-bimodule $P$ that is finitely generated projective by the right. Then, the functor $\Psi$ controls endolength. 
\end{lemma}

\begin{proof} Take $N\in \Gamma\g\mod$  and set $E:=\End_\Gamma(N)^{op}$. We know that for some natural number $n$, there is a surjective morphism $\Gamma^n\rightmap{}P$ of right $\Gamma$-modules. Therefore, we have a surjective morphism of right $E$-modules 
$$N^n\cong \Gamma^n\otimes_\Gamma N\rightmap{}P\otimes_\Gamma N.$$
Therefore, 
$\ell_E(P\otimes_\Gamma N)\leq n\times \ell_E(N)=n\times \Endol{N}$.  

Now, the action of $E$ on the right $E$-module $P\otimes_\Gamma N$ determined by the right $E$-module $N$ is such that $(p\otimes n)f=p\otimes f(n)$, for $p\in P$ and $n\in N$. The algebra $\End_Q(P\otimes N)^{op}$ acts on the right of $P\otimes_\Gamma N$ by the usual rule and determines a right action of $E$ by restriction using the isomorphism of algebras $E\rightmap{}\End_Q(P\otimes N)^{op}$ provided by the full and faithful functor $\Psi$. This action of $E$ on $P\otimes_\Gamma N$ coincides with the previous one and, hence, we obtain 
$\Endol{P\otimes_\Gamma N}\leq \ell_E(P\otimes_\Gamma N)$. 

If $0= N_t\subseteq \cdots\subseteq  N_1\subseteq N_0=N$ is a composition series for the $\End_\Gamma(N)^{op}$-module $N$, since $P$ is a projective right $\Gamma$-module, we have a filtration $0= P\otimes_\Gamma N_t\subseteq \cdots \subseteq P\otimes_\Gamma N_1\subseteq P\otimes_\Gamma N_0=P\otimes_\Gamma N$ of the $Q$-module  $P\otimes_\Gamma N$, which is in fact a filtration of the $\End_Q(P\otimes_\Gamma N)^{op}$-module  $P\otimes_\Gamma N$. Hence $\Endol{N}\leq \Endol{P\otimes_\Gamma N}$.  
\end{proof}

\bigskip
\noindent{\bf Proof of Theorem (\ref{main thm for algs with minimal's}):} As remarqued just before (\ref{R: Lambda tri implies D no es e-dscreto}), we can assume that $\Lambda$ is basic because Morita equivalences are representation embeddings which control endolength of indecomposables, see (\ref{L: endols y repembed fieles y plenos}).

From (\ref{R: Lambda tri implies D no es e-dscreto}), we can apply (\ref{T: reduccion a minimales}) to the Drozd's ditalgebra ${\cal D}$ of $\Lambda$ to obtain a quasi-minimal ditalgebra ${\cal Q}$, with minimal algebra $\Q$ of infinite representation type, and a composition of functors $F$ of the form 
$$\Q\g\mod
\rightmap{\LQ_{\Q}}{\cal Q}\g\mod\rightmap{G}{\cal D}\g\mod,$$
as in (\ref{T: reduccion a minimales}), which preserves indecomposability,  reflects isomorphism classes, and controls endolength of indecomposables.

 From \cite[(22.7)]{bsz}, we obtain that $GE(\Q)$ is a $D\g \Q$-bimodule which is finitely generated by the right and  that $F\cong L_{\cal D}(GE(\Q)\otimes_\Q-)$, where  the argument is the functor 
$GE(\Q)\otimes_\Q-:\Q\g\mod\rightmap{}D\g\mod$. Here $D$ denotes the subalgebra of the underlying tensor algebra of ${\cal D}$ formed by the elements of degree 0.

Consider the usual equivalence functor $\Xi_\Lambda:{\cal D}\g\mod\rightmap{}{\cal P}^1(\Lambda)$ and the cokernel funtor $\Cok:{\cal P}^1(\Lambda)\rightmap{}\Lambda\g\mod$, see \cite[\S18 and \S19]{bsz}. Denote by $Z_0=\Cok\Xi_\Lambda(D)$ the transition bimodule associated to $\Lambda$, as in \cite[(22.18)]{bsz}. Then, the $\Lambda\g \Q$-bimodule  $Z:=Z_0\otimes_DGE(\Q)$  is finitely generated by the right. Denote by $H$ the composition of functors 
$$\Q\g\mod\rightmap{F}{\cal D}\g\mod\rightmap{\Xi_\Lambda}{\cal P}^1(\Lambda)\rightmap{\Cok}\Lambda\g\mod.$$
Then, by \cite[(22.18)]{bsz}, we know that it is naturally isomorphic to 
$$\Cok\Xi_\Lambda L_{\cal D}(GE(\Q)\otimes_\Q-)\cong Z_0\otimes_DGE(\Q)\otimes_\Q-=Z\otimes_\Q-.$$
Now, if $\Q$ is a minimal algebra of the second type, we have, as in \cite[Claim 1 in the proof of (8.2)]{bps}, that every $M\in \Q\g\mod$ satisfies that $\Xi_\Lambda F(M)\in {\cal P}^2(\Lambda)$. Thus, in this case, the functor $H$ preserves indecomposables and reflects isomorphism classes.

If $\Q$ is a minimal algebra of the first type, we have as in \cite[Claim 2 in the proof of (8.2)]{bps}, that for any non-injective indecomposable $M\in \Q\g\mod$ we have $\Xi_\Lambda F(M)\in {\cal P}^2(\Lambda)$. So,  in this case $H$ preserves indecomposables which are not injective (since it is a composition of functors which preserve non-injectivity). Moreover, the restriction $\Q\g\mod_{\cal I}\rightmap{}\Lambda\g\mod$ of $H$, to the full subcategory of $\Q$-modules without injective direct summands,  reflects isomorphism classes. Indeed, if $M,N\in \Q\g\mod_{\cal I}$, then $\Xi_\Lambda F(M)$ and $\Xi_\Lambda F(N)$ have no injective direct summand. Thus, from \cite[(18.10)]{bsz}, we get that $H(M)\cong H(N)$ implies that $M\cong N$. 

 We know that $F$ controls endolength and,  from \cite[(4.4)]{bps}, we get that $\Cok \Xi_\Lambda$ controls endolength, then so does their composition $H$. 
$\hbox{ }\hfill\square$

\begin{remark} Under the assumptions of  theorem (\ref{main thm for algs with minimal's}), assume moreover that  $\Lambda$ is such that $\Lambda/\rad\Lambda$ is a finite product of copies of $k$. In this case, the minimal algebras appearing there  have their coefficient algebras isomorphic to $k$. So, there   
 is no such minimal algebra of the first type of infinite representation type, and the only  such minimal algebra of the second type of infinite representation type is $k[x]$. Then,  we get a functor 
 $H : k[x]\g\mod \rightmap{}\Lambda\g\mod$, which preserves indecomposables and reflects isomorphism classes, given
by tensoring with a bimodule $\,_\Lambda Z_{k[x]}$ which is finitely generated 
 over $k[x]$. From here we obtain a representation embedding $H':k[x]_h\g\mod \rightmap{}\Lambda\g\mod$, given by tensoring with a bimodule 
 $\,_\Lambda Z'_{k[x]_h}$ which is finitely generated projective 
 over $k[x]_h$. This is the case if the field $k$ is algebraically closed and $\Lambda$ is basic, see Thm. 6 of \cite{Bo1}. 
\end{remark}

\section{Reduction to principal ideal domains}

In this work, we use the term uniserial subcategory in the following sense. 

\begin{definition}\label{R: uniserial cats}
Let ${\cal U}$ be a full abelian subcategory of the category $\Sigma\g\mod$ of finite-dimensional  modules for some $k$-algebra $\Sigma$. Then ${\cal U}$ is called \emph{a uniserial subcategory of $\Sigma\g\mod$} iff every indecomposable module in ${\cal U}$ has a unique composition series in ${\cal U}$, and all the simple factors in  this composition series are isomorphic to one simple object of the category ${\cal U}$.
\end{definition}

\begin{lemma}\label{L: Gama-mod es prod de uniseriales con simple}
Let $\Gamma$ be a bounded principal ideal domain and let ${\cal P}$ be a complete set of representatives of the non-similar atoms of $\Gamma$. Then, the category of finite-length $\Gamma$-modules is a direct sum of uniserial subcategories 
$\{{\cal U}_p\}_{p\in {\cal P}}$, where each ${\cal U}_p$ contains a unique simple object $E_1^p$.

Moreover, we have almost split sequences of the form: 
$$\begin{matrix}\xi_1  : \  E^p_1  \rightmap{} E^p_2\rightmap{} E^p_1,\\
 \,\\
\zeta_n : \ E^p_n \rightmap{} E^p_{n+1}\oplus E^p_{n-1}\rightmap{}E^p_n, \hbox{ for } n \geq  2.\\
\end{matrix}$$
\end{lemma}
 
\begin{proof} The category of finite-length $\Gamma$-modules is well known. For instance, from \cite[(6.5)]{bps}, we have the following. 
For each atom $p \in \Gamma$, consider the corresponding simple $\Gamma$-module $S_p := \Gamma/\Gamma p$. Then, for each $i \in \hueca{N}$, up to isomorphism, there is a unique indecomposable $\Gamma$-module $E_i^p$ with length $i$ and all composition factors isomorphic to $S_p$. The family $\{E_i^p \mid i \in \hueca{N}, p \in {\cal P}\}$ is a complete set of representatives of the isoclasses of the indecomposable $\Gamma$-modules of finite length. They are all finite-dimensional.   
 
For $p\in {\cal P}$, consider the full subcategory ${\cal U}_p$ of $\Gamma\g\mod$  formed by the modules $M\in \Gamma\g\mod$ which admit a filtration 
$$0=M_0\subset M_{1}\subset\cdots \subset M_{n-1}\subset  M_n=M,$$ 
where $M_j/M_{j-1}\cong S_p$, for all $j\in [1,n]$. Thus, $\{E_i^p\mid i\in \hueca{N}\}$ are the representatives of  the indecomposables in ${\cal U}_p$. Moreover, in the proof of \cite[(6.5)]{bps}, it is shown that $\Hom_\Gamma(E_i^p,E_j^q)=0$, whenever $p$ and $q$ are not similar atoms, that is whenever $p\not= q$ in ${\cal P}$.

It remains to show that ${\cal U}_p$ is uniserial. Let $M$ be an indecomposable $\Gamma$-module in ${\cal U}_p$. Then, it  has a composition series 
$$0=M_0\subset M_{1}\subset\cdots \subset M_{n-1}\subset  M_n=M,$$ 
where $M_i/M_{i-1}\cong S_p$. 
Here $\ell_\Gamma(M_i)=i$, for all $i\in [1,n]$. 

\medskip
\emph{We will show that $M_j$ is the unique submodule of $M$ with length $j$.}
\medskip

We proceed by induction on the length $n=\ell_\Gamma(M)$. 

 If $n=1$, we have that $M=M_1$ is simple and our claim is obvious. 
 Assume that our claim holds for indecomposable modules in ${\cal U}_p$ with length $s$ such that  $s<n$.   
  We have an exact sequence 
 $$0\rightmap{}M_{n-1}\rightmap{f}M_n\rightmap{g}S_p\rightmap{}0.$$ 
 Let $S$ be a simple submodule of $M$ and $h:S\rightmap{}M_n$ the corresponding  inclusion. Since $M$ is indecomposable, we know that $gh:S\rightmap{}S_p$ is not an isomorphism, so $gh=0$ and $S\subseteq M_{n-1}$. Similarly, working with the given filtration of $M_{n-1}$, we show that $S\subseteq M_{n-2}$. In a finite number of steps, we get $S\subseteq M_1$.  Hence $S=M_1$. We have verified that the socle of $M$ is the simple $M_1$. 
 
 Consider now a submodule $N$ of $M_n$ with length $j>1$. Since $\soc(M)=M_1$, we know that $M_1=\soc(N)\subset N$. In particular, the submodules $M_1,\ldots,M_{n-1}$ are all indecomposable. We have the quotient $\eta:M_n\rightmap{}M_n/M_1$, 
and the submodules $\eta(M_j)=M_j/M_1$ and $\eta(N)=N/M_1$ with length $j-1$ of the module $M/M_1\in {\cal U}_p$, which has length $n-1$. 
\medskip

From \cite[(6.5)]{bps}, we know there is an irreducible morphism $g:M_n\rightmap{}M_{n-1}$ in $\Gamma\g\mod$, which must be surjective. Its kernel has length 1, so it is a simple submodule of $M_n$, and it coincides with $M_1$. It follows  that $M/M_1\cong M_{n-1}$ is indecomposable.  
Then, by induction hypothesis
 $N/M_1=M_j/M_1$ and, so, $M_j=N$, as we wanted to verify. 
\end{proof}

\begin{proposition}\label{P: preproj components discrete in heredit case}
Let $Q$ be a finite-dimensional hereditary algebra of infinite representation type. Then, its preprojective components, as well as its preinjective components, are $e$-discrete.  
\end{proposition}

\begin{proof} We only consider the case of a preprojective component ${\cal P}$, the other case is tackled dually. If   $M\in {\cal P}$,  there is a projective module $P_M\in {\cal P}$ and a non-negative number $\nu(M)$, such that $(DTr)^{\nu(M)}M=P_M$. We also know, that $\End_Q(M)\cong \End_Q(P_M)$ is a division algebra, see \cite[VIII.1.(1.5)]{ars}. Therefore, from (\ref{L: dim vs endol}), we obtain $\Endol{M}=\dim_kM/\dim_k\End_Q(P_M)$.  This implies that there is only a finite number of indecomposables in ${\cal P}$ with the same endolength iff there is only a finite number of indecomposables in ${\cal P}$ with the same dimension. That is ${\cal P}$ is $e$-discrete iff it is discrete. 

So, we have to show that there are only finitely many indecomposables in ${\cal P}$ with the same dimension. For this, we can assume that $Q$ is basic and consider a decomposition of the unit $1=\sum_{i=1}^ne_i$ of $Q$, as a sum of primitive orthogonal idempotents, denote by $\underline{\dim}(M)=(\dim_k e_1M,\ldots,\dim_k e_nM)$,  \emph{the dimension vector of} $M$. Recall that given $M,N\in {\cal P}$, we have that $M\cong N$ iff their their dimension vectors coincide, see \cite[VIII.2.(2.3)]{ars}. Then, consider the homological quadratic form $q$ of the hereditary algebra $Q$, which satisfies  
$$q(\underline{\dim}(M))=\dim_k\End_Q(M)-\dim_k\Ext_Q^1(M,M).$$
Since $M$ is preprojective, we have $\Ext_Q^1(M,M)=0$, see \cite[VIII.1.(1.7)]{ars}. Therefore, we get that for any $M\in {\cal P}$, $q(\underline{\dim}(M))\in \{d_1,\ldots,d_n\}$, where $d_i:=\dim_k\End_Q(Qe_i)$, for $i\in [1,n]$. There are only finitely many possible dimension vectors $\underline{\dim}(M)$ satisfying this, so ${\cal P}$ is discrete.   
\end{proof}

The following statement is similar to  \cite[(6.8)]{bps}. 

\begin{proposition}\label{L: minimales gen mansas vs DIP's} 
If $\Q$ is a generically tame minimal algebra of infinite representation type, then there is a bounded principal ideal domain $\Gamma$ which is not 
$e$-discrete and an exact full and faithful functor  $\Psi: \Gamma\g\Mod\rightmap{}\Q\g\Mod$. 
The functor $\Psi$ restricts to a representation embedding $\Psi':\Gamma\g\mod\rightmap{}\Q\g\mod_{\cal I}$.  
\end{proposition}

\begin{proof} This proof is an adaptation of the proof of \cite[(6.8)]{bps}. If
 $\Q$ is of the second type in (\ref{D: minimal algebra}). From \cite[(6.2)]{bps}, we can assume that $\Q = D[x, s]$, that is the skew polynomial algebra with coefficients in $D$, for some automorphism $s:D\rightmap{}D$. It is well known that $\Gamma := D[x, s]$ is a bounded principal ideal domain (see \cite[(3.15)]{J}) and 
taking as $\Psi$ the identity functor, we see, from (\ref{R: Gama=D[x,s] no e-discreto}), that our proposition holds in this case.

The generically tame minimal algebras $\Q$ of infinite representation type of the first type in (\ref{D: minimal algebra}) can be described as follows, see \cite[(6.6)]{bps}:

\begin{enumerate}
\item[(i)] $\Q$ is the matrix algebra $\begin{pmatrix}F&0\\ M&G\end{pmatrix}$, where $F$ and $G$ are finite-dimensional division $k$-algebras 
and $M$ is a simple $G\g F$-bimodule where the field $k$ acts centrally. Moreover, $\dim_G M = 2 = \dim M_F$; or
\item[(ii)] $\Q$ is the matrix algebra $\begin{pmatrix}F&0\\ M&G\end{pmatrix}$, where $F$ and $G$ are finite-dimensional division $k$-algebras 
and $M$ is a simple $G\g F$-bimodule where the field $k$ acts centrally.
 Moreover, $\dim_G M = 4$ and $\dim M_F =1$, or $\dim_G M =1$ and $\dim M_F =4$.
\end{enumerate}

In each one of these cases, it is shown in the proof of \cite[(6.8)]{bps}, using previous work by V. Dlab and M. Ringel  \cite{DR}\&\cite{DRMem}, and Crawley-Boevey \cite{CB0}, that there is a  bounded principal ideal domain $\Gamma$ and an exact full and faithful functor $\Psi:\Gamma\g\Mod\rightmap{}\Q\g\Mod$ mapping indecomposable $\Gamma$-modules of finite length onto regular $\Q$-modules. Moreover, $\Q\g\reg \cong 
\Psi(\Gamma\g\mod) \coprod {\cal U}$, where ${\cal U}$ is a  uniserial subcategory of $\Q\g\mod$ generated by a simple regular module. Furthermore, from (\ref{P: endol(Mn)=n endol(M1) en U unis}), 
 we see that, for each $d\in \hueca{N}$, there is, up to isomorphism, at most one indecomposable in ${\cal U}$, having endolength $d$.

As we already know (by \cite{ficlifts} or \cite[(8.6)]{CB3}), since $\Q$ has infinite representation type, it is not $e$-discrete. 
In fact, we know that $\Q\g\reg$ is not $e$-discrete, as a consequence of (\ref{P: preproj components discrete in heredit case}). Thus, there is some $d\in \hueca{N}$ and an infinite family of non-isomorphic indecomposable modules $\{\Psi(M(i))\}_i$ in $\Q\g\reg$ with    
endolength $\Endol{{\Psi(M(i))}}=d$. Since $\Psi$ controls endolength, see (\ref{L: endols y repembed fieles y plenos}), there is a subfamily $\{M(j)\}_j$ of $\{M(i)\}_i$ of non-isomorphic indecomposables in $\Gamma\g\mod$ with common endolength, and $\Gamma$ is not $e$-discrete.

The statement on the restriction of $\Psi$ is due to the following. If $N\in \Gamma\g\mod$ is such that $\Psi(N)$ admits an indecomposable injective direct summand $I$, say $\Psi(N)=I\oplus C$,  from \cite[(29.5)]{bsz}, we know that there is  an indecomposable direct summand $I'$ of $N$ such that $\Psi(I')\cong I$. Now, by the description of almost split sequences in $\Gamma\g\mod$, see \cite[(6.5)]{bps}, we know that $I'$ appears at the left of an almost split sequence $\xi$. Then, $\Psi(\xi)$ is an exact sequence in $\Q\g\mod$ which must split. Then, $\xi$ also splits, a contradiction.   So, the functor $\Psi$ restricts to a functor $\Psi':\Gamma\g\mod\rightmap{}\Q\g\mod_{\cal I}$, as wanted. 
\end{proof}

Notice that in the preceding proposition,  if $\Q$ is a minimal  algebra of type 2 or of type 1.(ii), the principal ideal domain $\Gamma$ is centrally bounded.

\medskip
If $\Q$ is a minimal algebra of the first type with infinite representation type,  then there are only two possibilities: either its quadratic form is positive semidefinite (so $\Q$ is tame hereditary) or its quadratic form is indefinite (so $\Q$ is wild hereditary). 

\begin{lemma}\label{L: Omega min wild tipo 1}
If $\Q$ is a wild hereditary minimal algebra of the first type, then there is a centrally bounded principal ideal domain $\Gamma$ which is not $e$-discrete and an exact full and faithful functor $\Psi:\Gamma\g\mod\rightmap{}\Q\g\mod_{\cal I}$.   
\end{lemma}

\begin{proof} From \cite[(8.2) and (8.4)]{CB3}, there is a finite field extension  $K$ of $k$ and a  $\Q\g K\langle  x, y\rangle$-bimodule $T$ which is finitely generated projective over $K\langle x,y\rangle$ and such
that the tensor product functor 
$$F:=T\otimes_{K\langle x, y\rangle}-: K\langle x, y\rangle\g\Mod \rightmap{}\Q\g\Mod$$ is full and faithful. Let $\phi:K\langle x,y\rangle\rightmap{}K[x]$ be the canonical projection and consider the  corresponding restriction functor $H:K[x]\g\Mod\rightmap{}K\langle x,y\rangle\g\Mod$. We have that $H\cong K[x]\otimes_{K[x]}-$, where the $K\langle x,y\rangle\g K[x]$-bimodule $K[x]$ is a left $K\langle x,y\rangle$-module by restriction through $\phi$. 
Then, the composition $FH:K[x]\g\Mod\rightmap{}\Q\g\Mod$ is exact, full and faithful, and can be realized as a tensor product by a bimodule which is finitely generated projective by the right.

Since $T$ is finitely generated by the right, the functor $F$ restricts to the subcategories of finite-dimensional modules, and so does the functor $H$, and their composition $FH:K[x]\g\mod\rightmap{}\Q\g\mod$. Take the principal ideal domain $\Gamma:=K[x]$ and notice, as in the proof of the last lemma,  that $\Im FH$ is contained in $\Q\g\mod_{\cal I}$, so $\Psi:=FH$ restricts to the desired exact full and faithful functor. Here $\Gamma$ is not $e$-discrete, for instance because it 
is particular case of (\ref{R: Gama=D[x,s] no e-discreto}). 
\end{proof}

\begin{lemma}\label{L: Omega min tipo 2} Assume that $\Q=T_D(V)$ is a minimal algebra of the second type. Then, there is a centrally bounded principal ideal domain $\Gamma$ which is not $e$-discrete and a full and faithful exact functor $\Psi:\Gamma\g\mod\rightmap{}\Q\g\mod$. 
\end{lemma}

\begin{proof} We just follow the proof of \cite[(6.2)]{bps}. There are two cases.
\medskip

\emph{Case 1:  $\dim_DV\geq 2$.} Here, by \cite[(8.2) and (8.5)]{CB3}, there is a finite field extension $K$ of $k$ and a $\Q\g K\langle x,y\rangle$-bimodule $Z$, which is free of finite rank over $K\langle x,y\rangle$ such that the tensor product $Z\otimes_{ K\langle x,y\rangle}-: K\langle x,y\rangle\g\Mod\rightmap{}\Q\g\Mod$ is full and faithful, and we proceed as in the proof of (\ref{L: Omega min wild tipo 1}) with $\Gamma=K[x]$, which is not $e$-discrete. 
\medskip

\emph{Case 2: $\dim_DV=1$.} Here, we know as mentioned in the proof of \cite[(6.2)]{bps}, that $\Q$ is a skew polynomial algebra $D[x,s]$, for some automorphism $s:D\rightmap{}D$. This case was already considered in (\ref{L: minimales gen mansas vs DIP's}). We can take $\Gamma=\Q$, which is a centrally bounded  principal ideal domain by \cite[(6.4)]{bps}.  
\end{proof}

\noindent{\bf Proof of Theorem (\ref{main thm for algs with PID's}):} 
By (\ref{main thm for algs with minimal's}), there is a minimal algebra $\Q$  of infinite representation type and a  functor  $H$ of one of the following types:
\begin{enumerate}
\item $H:\Q\g\mod_{\cal I}\rightmap{}\Lambda\g\mod$, when $\Q$ is a minimal algebra of the first type, or  
\item $H:\Q\g\mod\rightmap{}\Lambda\g\mod$, when $\Q$ is of the second type. 
\end{enumerate}
Here the functor $H$ is of the form $Z\otimes_\Q-$, for some $\Lambda\g \Q$-bimodule $Z$, which is finitely generated by the right.  

There are a few possibilities for $\Q$, considered in the preceding lemmas (\ref{L: minimales gen mansas vs DIP's}), (\ref{L: Omega min tipo 2}), and  (\ref{L: Omega min wild tipo 1}). In each case, there is a bounded principal ideal domain $\Gamma$, which is not $e$-discrete, and a functor $\Psi$ which is exact full and faithful, and we consider its composition $G:=H\Psi:\Gamma\g\mod\rightmap{}\Lambda\g\mod$. The functor  $G$ controls endolength of indecomposables because $\Psi$ and $H$ do so, by   (\ref{L: endols y repembed fieles y plenos}) and (\ref{main thm for algs with minimal's}). 
\hfill$\square$

\section{The representation embedding}

In this section we derive our main result (\ref{main exact thm for algs with PID's}), from the preceding work. For this, we consider  Ore localizations $S^{-1}\Gamma$, for some special denominator subset $S$ of a principal ideal domain $\Gamma$, see \cite[\S3.1]{Row}. 

\begin{remark}\label{R: sobre DIPs y DFUs} Recall that any principal ideal domain $\Gamma$ is a unique factorization domain in the following sense. Any non-zero non-unit element $a\in \Gamma$ is a product $a=p_1\cdots p_s$ where $p_1,\ldots,p_s$ are atoms of $\Gamma$ and for any other such factorization  $a=q_1\cdots q_t$ as product of atoms, we have $s=t$  and there is a permutation $\sigma\in S_t$ such that $q_i$ is similar to $p_{\sigma(i)}$, for each $i$. 
\end{remark}

\begin{lemma}\label{L: la restricción de localizar se porta bien}
Let $\Gamma$ be a bounded principal ideal domain and let $b \in \Gamma$  be  such that $ \Gamma b= b\Gamma \not=\Gamma$. Consider the denominator set $S = \{1, b, b^2, \ldots\}$. We have the epimorphism of rings 
$\nu : \Gamma \rightmap{}S^{-1}\Gamma$ 
and the corresponding restriction functor  $F : S^{-1}\Gamma\g \mod \rightmap{} \Gamma\g\mod$. Let $p$ be an atom of $\Gamma$ which is not similar to any of the atomic factors of $b$, then 
every  indecomposable $\Gamma$-module $M$ 
with composition factors isomorphic to  $\Gamma/\Gamma p$ is of the form $M\cong F(N)$, for some $N\in S^{-1}\Gamma\g\mod$.  
\end{lemma}

\begin{proof} Let $M\in \Gamma\g\mod$ be indecomposable with all composition factors isomorphic to $\Gamma/\Gamma p$.
We will show by induction on the length of $M$ that the linear map $\mu_b:M\rightmap{}M$ given by multiplication by $b$ is a linear isomorphism. This is enough, because we know that the structure of $M$ as a $\Gamma$-module is given by the morphism of $k$-algebras $\mu:\Gamma\rightmap{}\End_k(M)$, mapping each $a\in \Gamma$ onto the linear map $\mu_a$, given by multiplication by $a$. If $\mu_b$ is invertible, then $\mu_s$ is invertible for each $s\in S$. Thus by the universal property of localizations, we obtain a map $\hat{\mu}:S^{-1}\Gamma\rightmap{}\End_k(M)$, such that $\hat{\mu}\nu=\mu$, that is a structure of $S^{-1}\Gamma$-module $N$ with underlying vector space $M$. Thus, $F(N)= M$.

If we assume that  $M$ is simple, then $M = \Gamma m$ with $m\in M$ and, if $bm=0$, we would have $b\in\Gamma p$, which is not the case by (\ref{R: sobre DIPs y DFUs}).   Hence, 
$M=\Gamma m = \Gamma bm = b\Gamma m$. Then, the linear map $\mu_b: M\rightmap{}M$ given 
by multiplication by $b$ is surjective. Since $M$ is finite-dimensional, we obtain that $\mu_b$ is a linear isomorphism.   Then, our claim holds for indecomposable   $\Gamma$-modules of length 1. 

 Now, assume that the indecomposable module $M$ has length $l> 1$ and consider its composition series
$0 = M_0 \subset M_1  \subset  \cdots \subset M_l = M$ in $\Gamma\g\mod$.
By induction hypothesis, multiplication by $b$ determines linear automorphisms on each $M_i$, with $i<l$. From the proof of (\ref{L: Gama-mod es prod de uniseriales con simple}), we get an exact sequence of $\Gamma$-modules
$0 \rightmap{} M_1 \rightmap{} M \rightmap{} M_{l-1} \rightmap{} 0,$ where $M_{l-1}$ is indecomposable. 
Then, we have the following commutative diagram of linear maps
$$\begin{matrix}
0 &\rightmap{}& M_1& \rightmap{}& M& \rightmap{}& M_{l-1}& \rightmap{}& 0\\
&&\shortlmapdown{}&&\shortlmapdown{}&&\shortlmapdown{}&&\\
0 &\rightmap{}& M_1& \rightmap{}& M& \rightmap{}& M_{l-1}& \rightmap{}& 0,\\
\end{matrix}$$
where the vertical arrows are given by multiplication by $b$. 
By induction hypothesis, the left and right vertical arrows are linear isomorphisms, then so is the central one. 
\end{proof}

\medskip
\noindent{\bf Proof of theorem (\ref{main exact thm for algs with PID's}):} By (\ref{main thm for algs with PID's}),  there is a  bounded principal ideal domain $\Gamma$ which is not $e$-discrete and a functor  
 $$G:\Gamma\g\mod\rightmap{}\Lambda\g\mod,$$ 
 which preserves indecomposables, reflects isomorphism classes, and controls endolength.  The functor $G$ is of the form $Z\otimes_\Gamma-$, for some $\Lambda\g \Gamma$-bimodule $Z$ which is finitely generated by the right. 

The right $\Gamma$-module $Z$ decomposes as a direct sum $Z=T\oplus P$, where $T$ is a torsion finitely generated $\Gamma$-module and  $P$ is a free $\Gamma$-module of finite rank, see \cite[(1.4.4)]{Cohn2}. If $T=0$, we are done.
 
Assume that $T\not=0$, then  $T$ decomposes by \cite[(1.5.5)]{Cohn1} as a finite direct sum of indecomposables of the form $E_i^p$. Then, the annihilator of $T$  is a proper ideal of $\Gamma$ which is of the form $b\Gamma =\Gamma b$, for some non-unit element $b\in \Gamma$.   

We know that each $E_i^p$ has annihilator $q_p^i\Gamma$, where $q_p\Gamma$ is the annihilator of $E_1^p$. Then, $T$ is annihilated by the product of the elements $q_p^i$, corresponding to the indecomposables $E_i^p$ appearing in the decomposition of $T$. Thus,  $b\not=0$.    

Now, we consider the multiplicative subset $S=\{1,b,b^2,\ldots\}$ of $\Gamma$, which is a right denominator set and a left denominator set of $\Gamma$. 
The Ore localization $S^{-1}\Gamma\cong \Gamma S^{-1}$ is a bounded principal ideal domain.   The epimorphism  $\nu:\Gamma\rightmap{}S^{-1}\Gamma$ induces the full and faithful restriction functor $F:S^{-1}\Gamma\g\mod\rightmap{}\Gamma\g\mod$ which is isomorphic to $S^{-1}\Gamma\otimes-$.  Thus, the composition $GF:S^{-1}\Gamma\g\mod\rightmap{}\Lambda\g\mod$ preserves indecomposables,  reflects isomorphism classes, and is isomorphic to $Z\otimes_\Gamma S^{-1}\Gamma\otimes-$.  We claim that $Z\otimes_\Gamma S^{-1}\Gamma$ is a free right $S^{-1}\Gamma$-module. For this, it will be enough to show that $T\otimes_\Gamma S^{-1}\Gamma=0$.  Since $Tb=0$, and $b$ is invertible in $S^{-1}\Gamma$, we have $0=Tb\otimes_\Gamma S^{-1}\Gamma=T\otimes_\Gamma bS^{-1}\Gamma=T\otimes_\Gamma S^{-1}\Gamma$. 

 The principal ideal domains $S^{-1}\Gamma$ appearing before 
are not $e$-discrete. Indeed,  
 the atoms of $S^{-1}\Gamma$ correspond to those non-similar to the atomic factors of $b$.   
 Notice that, from  (\ref{L: la restricción de localizar se porta bien}), we obtain that  the restriction functor $F:S^{-1}\Gamma\g\mod\rightmap{}\Gamma\g\mod$ identifies the category $S^{-1}\Gamma\g\mod$ with the full  subcategory $\Gamma\g\mod_{\cal Z}$ of $\Gamma\g\mod$ formed by all the $\Gamma$-modules without direct summands isomorphic to  
$E_i^p$, with $i\in \hueca{N}$ and  
$p$ an atomic factor of $b$. Thus, since $\Gamma$ is not $e$-discrete, neither is $S^{-1}\Gamma$.   
\hfill$\square$

\begin{remark}\label{R: el caso de los reales}
In the case where  the ground field is the field $\hueca{R}$ of real numbers it is possible to make more explicit statements. In order to state them, we fix some notation: we consider the 
field $\hueca{C}$ of complex numbers and  the real quaternions $\hueca{H}$;  the skew polynomial algebras $\hueca{R}[x]$, $\hueca{C}[x]$, $\hueca{H}[x]$, and $\hueca{C}[x, \sigma]$, where $\sigma$ is the
complex conjugation; and the principal ideal domain $\hueca{D} = \hueca{R}[x, y]/\langle y^2 + x^2 + 1\rangle$. 

Let us describe the $\hueca{R}$-algebras which have a leading role in our arguments. This description relies heavily on the important work by Dlab and Ringel  on hereditary algebras, we  cite \cite{bps-reals}, but refer the reader to its bibliography. 
In this case, up to isomorphism, the generically tame 
finite-dimensional minimal algebras of infinite-representation type are 
$\begin{pmatrix}\hueca{R}&0\\
\hueca{H}&\hueca{H}\\ \end{pmatrix}$ and 
$\begin{pmatrix}\hueca{H}&0\\
\hueca{H}&\hueca{R} \\ \end{pmatrix}$, see \cite[(9.5)]{bps-reals}. The skew polynomial $\hueca{R}$-algebras, 
are, up to isomorphism, 
the following 
$\hueca{R}[x]$, $\hueca{C}[x]$, $\hueca{H}[x]$, and $\hueca{C}[x, \sigma]$. These are the minimal algebras $Q$  which may appear in (\ref{T: reduccion a minimales}) or (\ref{main thm for algs with minimal's}). The bounded  principal ideal domains that appear in theorem (\ref{main thm for algs with PID's}) are  $\hueca{R}[x]$, $\hueca{C}[x]$, $\hueca{H}[x]$, $\hueca{C}[x, \sigma]$, and 
 $\hueca{D}$.  Here, we can see \cite[(9.6)]{bps-reals} as an illustration of (\ref{L: minimales gen mansas vs DIP's}). The bounded principal ideal domains in (\ref{main exact thm for algs with PID's}),  are Ore localizations of the preceding ones. 
 
The wild hereditary case, addressed in (\ref{L: Omega min wild tipo 1}) and (\ref{L: Omega min tipo 2}), yields only the principal ideal domains  $\hueca{R}[x]$ and $\hueca{C}[x]$. The algebra $\hueca{H}[x]$ may only appear in (\ref{main thm for algs with minimal's}) if $\hueca{H}$ is  one of the coefficient algebras of $\Lambda$, see the proof of (\ref{T: reduccion a minimales})Case 2.a. 
\end{remark}

The following proposition is probably known, but we could not find any reference. 

\begin{proposition}\label{P: ideales max de Z y classes de similaridad de atomos}
Let $\Gamma$ be a principal ideal domain which is not a field and such that its  center $Z$ is again a principal ideal domain. Then if the set of maximal ideals of $Z$ has cardinality $C$,  any complete set ${\cal P}$ of representatives of the non-similar atoms of $\Gamma$ has at least the same cardinality $C$.    
\end{proposition}

\begin{proof} First notice that any non-trivial bilateral element $b \in \Gamma$ has the property that any  factor in a factorization of $b$ in $\Gamma$ is a left divisor and a right divisor of $b$. Indeed, if $b=ac$ in $\Gamma$, we have $bc=c'b$, for some $c'\in \Gamma$, thus 
$ac^2 = bc = c'b = c'ac$, and therefore, $b = ac = c'a$. Thus, any left divisor $a$ of $b$ is a right divisor of $b$. 
Similarly, one shows that any right divisor of $b$ is also a left divisor of $b$. 
Now, assume that $b = apc$, a factorization in $\Gamma$. As before, we obtain $ b = (pc) a'$ and $b = c'(ap)$, for some $a',c'\in \Gamma$, as claimed.  

Now, consider the collection ${\cal M}(\Gamma)$ of maximal (two sided) ideals of $\Gamma$ and let us exhibit an injective map $\alpha:{\cal M}(\Gamma)\rightmap{}{\cal P}$. Given $\Gamma b\in {\cal M}(\Gamma)$,   
by \cite[(1.2.19)]{J}, the atomic factors of $b$ belong to the same similarity class. We can assume that some atomic factor $p$ of $b$ belongs to ${\cal P}$ and set $\alpha(\Gamma b)=p$. 
As remarked before, $p$ is a right factor of $b$, so $\Gamma b\subseteq \Ann_\Gamma(\Gamma/\Gamma p)$. Since $\Gamma/\Gamma p\not=0$, we get $\Gamma b= \Ann_\Gamma(\Gamma/\Gamma p)$. 

Assume that  $\Gamma b'$ is a maximal ideal of $\Gamma$ different from $\Gamma b$ and $p'$ is an atomic  factor of $b'$ with $p$ similar to $p'$. Thus $\Gamma/\Gamma p\cong \Gamma/\Gamma p'$, and their annihilators coincide, which is not possible. So, $\alpha$   is an injective map. 

It remains to show that there is an injection $\beta:{\cal M}(Z)\rightmap{}{\cal M}(\Gamma)$, where ${\cal M}(Z)$ denotes the collection of maximal ideals of $Z$. We may assume that the principal ideal domain $Z$ is not a field. 
Assume that $Zz$ is a maximal ideal of $Z$. Then, $z$ has a factorization $z=a_1a_2\cdots a_s$, where $a_1,\ldots,a_s\in \Gamma$ are bilateral elements and  $\Gamma a_s$ is a maximal ideal of $\Gamma$. We choose such factor $a_s$ of $z$ and define $\beta(Zz)=\Gamma a_s$. Now, assume that $Zz'$ is another maximal ideal of $Z$ with $Zz'\not=Zz$, factorize $z'=a'_1a'_2\cdots a'_t$ as before and let us see that $\Gamma a_s\not=\Gamma a'_t$.  
If $\Gamma a_s=\Gamma a'_t$, then  $a_s=ua'_t$, for some unit $u\in \Gamma$. By the first paragraph of this proof, $a_s$ divides by the right $z$ and $z'$. But since $Zz+Zz'=Z$, we have $1=cz+c'z'$ which implies that $a_s$ divides $1$, a contradiction. 
\end{proof}

\begin{lemma}\label{L: cotas para endol de Gama-simples} Let $\Gamma$ be a principal ideal domain with center $Z$. Assume that $\Gamma$ is finitely generated as a $Z$-module, say by $m$ elements. Let $\Gamma b$ be a maximal ideal of $\Gamma$ and set $I=Z\cap \Gamma b$. Then, for any atom $p\in \Gamma$ dividing $b$, we have  
$$\Endol{\Gamma/\Gamma p}\leq m\times \ell_{Z}(Z/I),$$
where $\ell_Z$ denotes the length as a $Z$-module.  
\end{lemma}

\begin{proof} There is a surjective morphism of $Z$-modules $Z^m\rightmap{}\Gamma$, which induces a surjective morphism $(Z/I)^m\rightmap{}\Gamma/\Gamma b$, so $\ell_Z(\Gamma/\Gamma b)\leq m\times \ell_Z(Z/I)$. 

The morphism of algebras $\rho: Z\rightmap{}\End_\Gamma(\Gamma/\Gamma b)^{op}$ which maps any $z$ onto the right multiplication by $z$ on $\Gamma/\Gamma b$ induces on $\Gamma/\Gamma b$ the natural action of $Z$ on $\Gamma/\Gamma b$. This implies that any composition series for the $\End_\Gamma(\Gamma/\Gamma b)^{op}$-module $\Gamma/\Gamma b$ is a $Z$-module filtration  of the $Z$-module $\Gamma/\Gamma b$. Therefore, we have  $\Endol{\Gamma/\Gamma b}\leq \ell_Z(\Gamma/\Gamma b)$. 

Since $\Gamma b$ is a maximal ideal of $\Gamma$, we get that $\Gamma/\Gamma b$ is a simple ring. As recalled before, all atomic factors of $b$ are similar. Thus $\Gamma/\Gamma b$ has a simple left ideal.   
Hence, $\Gamma/\Gamma b\cong M_n(D)$, for some $n\in \hueca{N}$ and  a finite-dimensional division $k$-algebra $D$. Moreover, we can take  $D=\End_\Gamma(S)^{op}$, where $S$ is, up to isomorphism, the unique simple left $\Gamma/\Gamma b$-module. Here,  $S=M_{1\times n}(D)$ is a simple left $M_n(D)$-module with the usual matrix multiplication. So there is an isomorphism $D^n\rightmap{}S$ of right $D$-modules.

 As before, we have $\Endol{\,_{\Gamma/\Gamma b} S}\leq \ell_D(S_D)= n$. Since $p$ is an atom, the left $\Gamma/\Gamma b$-module $\Gamma/\Gamma p$ is  simple and we have $\Gamma/\Gamma p\cong S$. We have $\Gamma/\Gamma b$ is isomorphic to $S^n$ and we know that for a finite direct sum, the endolength of the summands is always bounded by the endolength of the sum.   Then, 
$$\Endol{\Gamma/\Gamma p}=\Endol{S}
\leq \Endol{\Gamma/\Gamma b}
\leq \ell_Z(\Gamma/\Gamma b)
 \leq m\times \ell_Z(Z/I).$$
\end{proof}

\begin{remark}\label{R: Gama=D[x,s] no e-discreto} It is known that the principal ideal domain $\Gamma=D[x,s]$ of twisted polynomials over a finite-dimensional division algebra is not $e$-discrete, see \cite[(8.6)]{CB3}. This can also be derived from the preceding two statements as follows.

  From \cite[(1.1.22)]{J}, the center of $\Gamma$  has the form $Z=F[t]$, a polynomial algebra over a finite field extension $F$ of $k$, and $\Gamma$ is finitely generated over $Z$. So $Z$ is a principal ideal domain with infinitely many prime ideals. 
  
  Recall from,  the proof of (\ref{P: ideales max de Z y classes de similaridad de atomos}), that $\Gamma$ admits an infinite set of non-similar atoms ${\cal P}$. Each $p\in {\cal P}$ appears as an atomic factor of some two-sided element $b\in \Gamma$ such that $\Gamma b$ is a maximal ideal of $\Gamma$. Since $\Gamma b$ is maximal, it is a prime ideal of $\Gamma$, so $I=Z\cap \Gamma b$ is a prime (thus maximal) ideal of $Z$, and $Z/I$ is a simple $Z$-module. Therefore, from (\ref{L: cotas para endol de Gama-simples}), we obtain that  $\Endol{\Gamma/\Gamma p}\leq m\times \ell_Z(Z/I)=m$, where $m$ is the number of generators of the $Z$-module $\Gamma$.  

Thus, there are infinitely many simple $\Gamma$-modules of the form $\Gamma/\Gamma p$, which are in fact all finite-dimensional, with the same endolength, and $\Gamma$ is not $e$-discrete. 
\end{remark}

\section{The property EBTII}\label{S: EBTII}

In this section, we study some useful relations between the endolengths of indecomposable modules filtered by the same simple object and the property EBTII for bounded principal ideal domains.

\begin{proposition}\label{P: KM de inesc en U uniserial} Assume that ${\cal U}$ is a uniserial subcategory of $\Sigma\g\mod$ for some $k$-algebra $\Sigma$ and that we have a family  $\{M_n \mid n \in \hueca{N}\}$ of representatives 
of the isoclasses of the indecomposable $\Sigma$-modules in ${\cal U}$. Moreover, assume that we have almost split sequences in the category ${\cal U}$ of the form: 
$$\begin{matrix}\xi_1  : \  M_1  \rightmap{} M_2\rightmap{} M_1,\\
 \,\\
\zeta_n : \ M_n \rightmap{} M_{n+1}\oplus M_{n-1}\rightmap{}M_n, \hbox{ for } n \geq  2.\\
\end{matrix}$$
Set $E_{M_n}=\End_\Sigma(M_n)^{op}$ and $K_{M_n}= \End_\Sigma(M_n)^{op}/\Rad \End_\Sigma(M_n)^{op}$.  Then,  there is an isomorphism of $k$-algebras $K_{M_n}\cong K_{M_1}=E_{M_1}$, for all $n\in \hueca{N}$. 
\end{proposition}

\begin{proof}  Assume that $n\geq 2$. From our assumptions, we get that  there is a  chain of injective irreducible morphisms
$$M_1\rightmap{\alpha_1}M_2\rightmap{}\cdots\rightmap{}M_{n-1}\rightmap{\alpha_{n-1}}M_n$$
and a chain of surjective irreducible morphisms
$$M_n\rightmap{\beta_{n-1}}M_{n-1}\rightmap{}\cdots\rightmap{}M_2\rightmap{\beta_1}M_1,$$
such that 
$\beta_1\alpha_1=0;\alpha_1\beta_1=\beta_2\alpha_2;\,\cdots; \beta_i\alpha_i=\alpha_{i-1}\beta_{i-1};\cdots$.   Since $M_n$ is uniserial,  we have that each $f\in \End_\Sigma(M_n)$ induces a commutative diagram
$$\begin{matrix}
0&\rightmap{}&M_{n-1}&\rightmap{\alpha_{n-1}}&M_n&\rightmap{\hat{\beta}_n}&M_1&\rightmap{}&0\\
&&\shortrmapdown{f_{n-1}}&&\shortrmapdown{f}&&\shortrmapdown{f_1}&&\\
0&\rightmap{}&M_{n-1}&\rightmap{\alpha_{n-1}}&M_n&\rightmap{\hat{\beta}_n}&M_1&\rightmap{}&0\\
\end{matrix}$$
where $\hat{\beta}_n=\beta_1\beta_2\cdots\beta_{n-1}$. The recipe $f\longmapsto f_1$ defines a morphism of algebras $\varphi_n:E_{M_n}\rightmap{}E_{M_1}$. Since $E_{M_1}$ is a division algebra and $\Rad E_{M_n}$ consists of nilpotent morphisms, the map $\varphi_n$ induces an injective morphism of division algebras
$\overline{\varphi}_n:K_{M_n}\rightmap{}K_{M_1}$. It only remains to show that $\overline{\varphi}_n$ is surjective or, equivalently, that $\varphi_n$ is surjective.  We proceed to a proof by induction on $n$.

If $n=2$, we have an almost split sequence
$$\begin{matrix}0&\rightmap{}&M_1&\rightmap{\alpha_1}&M_2&\rightmap{\beta_1}&M_1&\rightmap{}&0.\\
\end{matrix}$$
Given any morphism $f_1:M_1\rightmap{}M_1$, the composition $f_1\beta_1:M_2\rightmap{}M_1$ is not an isomorphism. Hence, there is a morphism $f:M_2\rightmap{}M_2$ such that $\beta_1f=f_1\beta_1$. So $\varphi_2$ is surjective.

Assume that $n\geq 3$ and that $\varphi_{n-1}$ is surjective, and let us show that $\varphi_n$ is surjective. Consider the almost split sequence 
$$\begin{matrix} 0&\rightmap{}&M_{n-1}&\rightmap{ \ (\alpha_{n-1},\beta_{n-2})^t \ }&M_n\oplus M_{n-2}&\rightmap{ \ (\beta_{n-1},\alpha_{n-2}) \ }&M_{n-1}&\rightmap{}&0.\\
\end{matrix}$$
Given $f_1\in E_{M_1}$, by induction hypothesis, there is some $f_{n-1}\in E_{M_{n-1}}$ with $f_1\hat{\beta}_{n-1}=\hat{\beta}_{n-1} f_{n-1}$. Consider the composition $f_{n-1}\beta_{n-1}:M_n\rightmap{}M_{n-1}$, which is not an isomorphism, thus there is a morphism $h:M_n\rightmap{}M_n\oplus M_{n-2}$ such that $(\beta_{n-1},\alpha_{n-2})h =f_{n-1}\beta_{n-1}$. Denote the matrix components of $h$  as $h=(f_n,g)^t$. So, we get $\beta_{n-1} f_n+\alpha_{n-2} g=f_{n-1}\beta_{n-1}$. Therefore, we have 
$$\hat{\beta}_{n-1}\beta_{n-1}f_n+\hat{\beta}_{n-1}\alpha_{n-2}g=\hat{\beta}_{n-1}f_{n-1}\beta_{n-1}=f_1\hat{\beta}_{n-1}\beta_{n-1}=f_1\hat{\beta}_n.$$
Moreover, we have 
$$\hat{\beta}_{n-1}\alpha_{n-2}g
=
\beta_1\cdots\beta_{n-2}\alpha_{n-2}g
=
\beta_1\cdots \beta_{n-3}\alpha_{n-3}\beta_{n-3}g
=
\beta_1\alpha_1\beta_1\cdots\beta_{n-3}g
=
0.$$
Finally, we obtain $\hat{\beta}_nf_n=\hat{\beta}_{n-1}\beta_{n-1}f_n=f_1\hat{\beta}_n$, so $\varphi_n$ is surjective. 
\end{proof}

Given a bounded principal ideal domain $\Gamma$, we can apply this proposition  to the uniserial subcategories ${\cal U}_p$ which form $\Gamma$-mod in (\ref{L: Gama-mod es prod de uniseriales con simple}), see \cite[(6.5)]{bps}. So we have the following. 

\begin{corollary}\label{P: KM de inesc de Gamma-mod}
Let $\Gamma$ be a bounded principal ideal domain. For any  indecomposable 
finite-dimensional $\Gamma$-module $M$, set $E_M=\End_\Gamma(M)^{op}$ and $K_M= \End_\Gamma(M)^{op}/\Rad \End_\Gamma(M)^{op}$. Let  $S$ denote the simple socle of $M$, see (\ref{L: Gama-mod es prod de uniseriales con simple}). Then,  there is an isomorphism of $k$-algebras $K_M\cong K_S=E_S$. 
\end{corollary}

\begin{proposition}\label{P: endol(Mn)=n endol(M1) en U unis}
Assume that ${\cal U}$ is a uniserial  subcategory of $\Sigma\g\mod$ satisfying the assumptions of (\ref{P: KM de inesc en U uniserial}). Then, if $M_n$ is an indecomposable $\Sigma$-module in ${\cal U}$ with length $n$ and socle $M_1$, we have 
$\Endol{M_n}=n\times \Endol{M_1}$. 

Thus, given $d\in \hueca{N}$, the category ${\cal U}$ has at most one indecomposable with endolength $d$.  
\end{proposition}

\begin{proof} From the last proposition, we know that $\dim_kK_{M_n}=\dim_kK_{M_1}$, Then, from (\ref{L: dim vs endol}), we have 
$$\begin{matrix}
\Endol{M_n}&=&\dim_k M_n/\dim_kK_{M_n}\hfill\\
&=& 
(n\times \dim_kM_1)/\dim_kK_{M_1}\hfill\\
&=&
n\times \dim_kM_1/\dim_kK_{M_1}=n\times \Endol{M_1}.\hfill\\
\end{matrix}$$
\end{proof}

\begin{corollary}\label{C: endol de inesc vs endol de simple en U uniserial}
Assume that $\widetilde{\cal U}$ is a full subcategory of $\Sigma\g\mod$, for some $k$-algebra $\Sigma$, such that $\widetilde{\cal U}=\coprod_{p\in {\cal P}}{\cal U}_p$, where each ${\cal U}_p$ is a uniserial subcategory satisfying the assumptions of (\ref{P: KM de inesc en U uniserial}). Then,  if $\widetilde{\cal U}$ is not $e$-discrete, the index set ${\cal P}$ is infinite and $\widetilde{\cal U}$  satisfies EBTII. 
\end{corollary}

\begin{proof} If the subcategory $\widetilde{\cal U}$ is not $e$-discrete, there is some number $d$ and an infinite family $\{M(i)\}_i$ of non-isomorphic finite-dimensional  indecomposable  $\Sigma$-modules in $\widetilde{\cal U}$ with common endolength $\Endol{M(i)}=d$. Denote by $S(i)$ the simple socle of each $\Sigma$-module $M(i)$ and by $n_i$ its length. For all $i$, we have 
$$d=\Endol{M(i)}=n_i\times \Endol{S(i)},$$ and the endolength of all these simple modules is bounded by $d$. If $S(i)\cong S(j)$, for some $i$, we would have 
$$n_i\times \Endol{S(i)}=\Endol{M(i)}=d=\Endol{M(j)}=n_j\times \Endol{S(j)}.$$
Thus, $M(i)$ and $M(j)$ have the same length and the same socle, which implies that they are isomorphic, thus $i=j$.  So, there is some $d_1<d$ and an infinite family of non-isomorphic simple $\Sigma$-modules $S(i)$ with the same endolength $d_1$.

Hence, for each one of these simples $S(i)$, the indecomposable module $N_n(i)$ with length $n$ and socle $S(i)$ has endolength $\Endol{N_n(i)}=n\times d_1$. The infinite family of non-isomorphic finite-dimensional indecomposable $\Sigma$-modules $\{N_n(i)\}_i$  in $\widetilde{\cal U}$ with endolength $n\times d_1$ can be  constructed for any $n$. 
\end{proof}

\begin{corollary}\label{C: endol de inesc vs endol del soclo en Gamma-mod} Let $\Gamma$ be a bounded principal ideal domain.  If $M_n$ is the indecomposable $\Gamma$-module with length $n$ and socle $M_1$, we have
$$\Endol{M_n}=n\times \Endol{M_1}.$$
Moreover, if $\Gamma$ is not $e$-discrete  it satisfies satisfies EBTII. 
\end{corollary}

\begin{proof} It follows from \cite[(6.5)]{bps}, (\ref{C: endol de inesc vs endol de simple en U uniserial}), and (\ref{P: KM de inesc de Gamma-mod}). 
\end{proof}

\begin{lemma}\label{L: EBTII se traslada con funtores que controlan endolength} Assume that a functor  $G:\Sigma\g\mod\rightmap{}\Lambda\g\mod$  preserves indecomposables, reflects isomorphism classes, and controls endolength of indecomposables. 
Then, if  EBTII holds for  $\Sigma\g\mod$, it also holds for  $\Lambda\g\mod$. 
\end{lemma}

\begin{proof} By assumption, there are $c,c'\in \hueca{N}$ such that, for any indecomposable $M\in \Sigma\g\mod$, we have  
$$   \Endol{M}\leq c\times\Endol{GM} \hbox{ \ and \ } \Endol{GM}\leq c'\times\Endol{M},$$
and there are infinitely many endolengths $d$'s such that for each one of these $d$'s there are infinitely many non-isomorphic finite-dimensional indecomposable $\Sigma$-modules with the same  endolength $d$. 

If we start with such an infinite family  
$\{M_i\}_i$ in $\Sigma\g\mod$ with $\Endol{M}=d_1$, we obtain an infinite subfamily  $\{GM_j\}_j$ of the family of images of $\{GM_i\}_i$ with the same endolength $\Endol{GM_j}=d'_1\leq c'd_1$, because $\Endol{GM_i}\leq c'\times \Endol{M_i}$.

Now, by assumption, we can choose a sufficiently large endolength $d_2$ (say with $d_2>c d'_1$) for which there is an infinite family of non-isomorphic finite-dimensional indecomposables  $\{N_i\}_i$ in  $\Sigma\g\mod$  with $\Endol{N_i}=d_2$. 
As before, we obtain an infinite subfamily  $\{GN_j\}_j$ of the family of images $\{GN_i\}_i$ with common endolength  $\Endol{GN_j}=d'_2\leq c'd_2$, again because $\Endol{GN_i}\leq c'\times \Endol{N_i}$. Since $G$ controls endolength of indecomposables, we obtain  
$$cd'_1<d_2=\Endol{N_i}\leq c\times \Endol{GN_i}=cd'_2.$$
Thus, $d'_1<d'_2$. We can iterate the preceding procedure to construct inductively endolengths $d'_1<d'_2<\cdots$ and families of non-isomorphic finite-dimensional indecomposables in $\Lambda\g\mod$ with the same endolength $d'_i$, for each $d'_i$.   
\end{proof}

 \noindent{\bf Proof of Theorem (\ref{T: EBTII}):}   From (\ref{main exact thm for algs with PID's}), we have  a  representation embedding $G:\Gamma\g\mod\rightmap{}\Lambda\g\mod$, where $\Gamma$ is a bounded principal ideal domain, which is not $e$-discrete.   The algebra $\Gamma$ satisfies EBTII as a consequence of (\ref{C: endol de inesc vs endol del soclo en Gamma-mod}). 
Our theorem follows from (\ref{L: EBTII se traslada con funtores que controlan endolength}), because $G$  controls  endolength of indecomposables. \hfill$\square$
 \bigskip

 In the following proposition we show that the formula relating the endolength of indecomposables in $\Gamma\g\mod$ described in (\ref{C: endol de inesc vs endol del soclo en Gamma-mod}), is preserved by 
 the exact embedding $G:\Gamma\g\mod\rightmap{}\Lambda\g\mod$ of (\ref{main exact thm for algs with PID's}).  

 \begin{proposition}\label{P: endolength formula} Let $G:\Gamma\g\mod\rightmap{}\Lambda\g\mod$ be the representation embedding of (\ref{main exact thm for algs with PID's}). Consider a family of indecomposable $\Gamma$-modules $\{M_n\}_{n\in \hueca{N}}$ such that each $M_n$ has socle $M_1$ and length $n$. Then, we have 
 $$\Endol{G(M_n)}=n\times \Endol{G(M_1)}.$$  
 \end{proposition}

\begin{proof} By construction, the functor $G$ is the following composition
$$\Gamma\g\mod\rightmap{\Psi}\Q\g\mod\rightmap{\LQ_{\Q}}
{\cal Q}\g\mod\rightmap{F}
{\cal D}\g\mod\rightmap{\Xi}{\cal P}^1(\Lambda)\rightmap{\Cok}
\Lambda\g\mod,$$
where ${\cal Q}$ is a quasi-minimal ditalgebra, $\Q$ is its minimal algebra, $\Psi$ is a full and faithful functor, $F$ is a composition of reduction functors, so it is full and faithful, and $\Cok$ is a full functor. 
\medskip

\noindent\emph{Claim 1: There is an isomorphism of $k$-algebras $K_{G(M_n)}\cong K_{G(M_1)}$.}
\medskip

We know that $\End_\Gamma(M_n)^{op}=K_{M_n}\oplus \Rad\End_\Gamma(M_n)^{op}$. 
Since $\Psi$ is full and faithful, we have $\End_\Q(\Psi(M_n))^{op}=\Psi(K_{M_n})\oplus \Rad\End_\Q(\Psi(M_n))^{op}$,  for all $n$. Moreover, from 
 (\ref{R: sobre quasi-dits}), we know that  $\End_{\cal Q}(\LQ_{\Q}(\Psi(M_n)))^{op}=\LQ_{\Q}(K_{\Psi(M_n)})\oplus \Rad\End_{\cal Q}(\LQ_{\Q}(\Psi(M_n)))^{op}$. 
Since the functors $F$ and $\Xi$ are full and faithful, we get 
$$\End_{{\cal P}^1(\Lambda)}(\Xi F\LQ_{\Q}(\Psi(M_n)))^{op}=\Xi F \LQ_{\Q} \Psi(K_{M_n})\oplus \Rad\End_{{\cal P}^1(\Lambda)}(\Xi F\LQ_{\Q}\Psi(M_n))^{op}.$$
Using that   the functor $\Cok$ is full, we obtain 
$$\End_\Lambda(G(M_n))^{op}=\Cok\Xi F\LQ_{\Q}\Psi(K_{M_n})\oplus \Rad\End_{\Lambda}(\Cok\Xi F\LQ_{\Q}\Psi(M_n))^{op}.$$
Hence, $\Cok \Xi F \LQ_{\Q}\Psi(K_M)$ is a division algebra isomorphic to $K_{G(M_n)}$. Thus, for each $n$, we have  $K_{M_n}\cong  K_{G(M_n)}$, for all $n$.  From (\ref{P: KM de inesc de Gamma-mod}), we know that $K_{M_n}\cong K_{M_1}$. Thus the formula of our Claim 1 holds.
\medskip

\noindent\emph{Claim 2: We have $\dim_kG(M_n)=n\times \dim_kG(M_1)$.}
\medskip

Indeed, for $n\geq 2$, we have exact sequences in $\Gamma\g\mod$ of the form 
$$0\rightmap{}M_{n-1}\rightmap{}M_n\rightmap{}M_1\rightmap{}0,$$
which are mapped by the exact functor $G$ on the exact sequences 
$$0\rightmap{}G(M_{n-1})\rightmap{}G(M_n)\rightmap{}G(M_1)\rightmap{}0$$
in $\Lambda\g\mod$. The the announced formula follows by an easy induction from the additivity of dimension for exact sequences.

Finally, using the preceding Claim 1, Claim 2, and (\ref{L: dim vs endol}), we have 
$$\begin{matrix}
\Endol{G(M_n)}
&=&
\dim_kG(M_n)/\dim_kK_{G(M_n)}\hfill\\
&=&
n\times \dim_kG(M_1)/\dim_kK_{G(M_1)}\hfill\\
&=&
n\times \Endol{G(M_1)}.\hfill\\
\end{matrix}$$ 
\end{proof}

\hskip2cm

\vbox{\noindent R. Bautista\\
Centro de Ciencias Matem\'aticas\\
Universidad Nacional Aut\'onoma de M\'exico\\
Morelia, M\'exico\\
raymundo@matmor.unam.mx\\}

\vbox{\noindent E. P\'erez\\
Facultad de Matem\'aticas\\
Universidad Aut\'onoma de Yucat\'an\\
M\'erida, M\'exico\\
jperezt@correo.uady.mx\\}

\vbox{\noindent L. Salmer\'on\\
Centro de Ciencias  Matem\'aticas\\
Universidad Nacional Aut\'onoma de M\'exico\\
Morelia, M\'exico\\
salmeron@matmor.unam.mx\\}

\end{document}